\documentclass[12pt]{amsart}
\usepackage{amsmath}
\usepackage{amsthm}
\usepackage{amssymb}
\usepackage{graphics}
\usepackage{bm}
\usepackage{setspace}

\usepackage{amsfonts}
\usepackage{eucal}
\usepackage{latexsym}
\usepackage{mathdots}
\usepackage{mathrsfs}
\usepackage{enumitem}

\usepackage{fullpage}
\usepackage{color}

\newcommand{\be}{\begin{equation}}
\newcommand{\ee}{\end{equation}}
\newcommand{\ba}{\begin{eqnarray}}
\newcommand{\ea}{\end{eqnarray}}
\newcommand{\bal}{\begin{align}}
\newcommand{\eal}{\end{align}}
\newcommand{\baln}{\begin{align*}}
\newcommand{\ealn}{\end{align*}}
\newcommand{\bi}{\begin{itemize}}
\newcommand{\ei}{\end{itemize}}
\newcommand{\bn}{\begin{enumerate}}
\newcommand{\en}{\end{enumerate}}
\newcommand{\bbm}{\begin{bmatrix}}
\newcommand{\ebm}{\end{bmatrix}}
\newcommand{\bpm}{\begin{pmatrix}}
\newcommand{\epm}{\end{pmatrix}}
\newcommand{\bp}{\begin{proof}}
\newcommand{\ep}{\end{proof}}
\newcommand{\nn}{\nonumber}

\newcommand{\mr}{\ensuremath{\mathrm}}
\newcommand{\scr}{\ensuremath{\mathscr}}
\newcommand{\mbf}{\ensuremath{\mathbf}}
\newcommand{\mc}{\ensuremath{\mathcal}}
\newcommand{\mf}{\ensuremath{\mathfrak}}
\newcommand{\ov}{\ensuremath{\overline}}

\newcommand{\wt}{\ensuremath{\widetilde}}

\newcommand{\ga}{\ensuremath{\gamma}}
\newcommand{\Om}{\ensuremath{\Omega}}

\newcommand{\La}{\ensuremath{\Lambda }}
\newcommand{\la}{\ensuremath{\lambda }}

\newcommand{\eps}{\ensuremath{\epsilon }}

\def\C{\mathbb{C}}

\def\D{\mathbb{D}}

\def\N{\mathbb{N}}
\def\B{\mathbb{B}}

\def\A{\mathcal{A} _d}

\def\fp{\mathbb{C} \{ \mathfrak{z} _1 , ... , \mathfrak{z} _d \} }

\renewcommand{\H}{\ensuremath{\mathcal{H} }}
\newcommand{\J}{\ensuremath{\mathcal{J} }}

\newcommand{\K}{\ensuremath{\mathcal{K} }}
\newcommand{\F}{\ensuremath{\mathbb{F} }}

\newcommand{\verteq}{\rotatebox{90}{$\,=$}}

\newcommand{\ip}[2]{\ensuremath{\langle {#1} , {#2} \rangle}}
\newcommand{\ipcn}[2]{\ensuremath{\left( {#1} , {#2} \right) _{\C ^n}}}

\newcommand{\dom}[1]{\ensuremath{\mathrm{Dom} ({#1}) }}

\newcommand{\ran}[1]{\ensuremath{\mathrm{Ran} \left( {#1} \right) }}

\renewcommand{\ker}[1]{\ensuremath{\mathrm{Ker} ({#1}) }}

\newcommand{\re}[1]{\ensuremath{\mathrm{Re} \left( {#1} \right) }}

\numberwithin{equation}{section}
\numberwithin{subsection}{section}

\newtheorem{thm}[subsection]{Theorem}
\newtheorem{claim}[subsection]{Claim}

\newtheorem{lemma}[subsection]{Lemma}
\newtheorem{prop}[subsection]{Proposition}
\newtheorem{cor}[subsection]{Corollary}
\newtheorem*{thm*}{Theorem}

\theoremstyle{definition}
\newtheorem{defn}[subsection]{Definition}
\newtheorem{remark}[subsection]{Remark}

\title[NC Lebesgue Decomposition]{Lebesgue Decomposition of Non-commutative measures}

\author{Michael T. Jury}
\address{University of Florida}
\email{mtjury@ad.ufl.edu}

\author{Robert T.W. Martin}
\address{University of Manitoba}
\email{Robert.Martin@umanitoba.ca}

\thanks{First named author partially supported by NSF grant DMS-1900364.}

\begin{document}
\small

\bibliographystyle{unsrt}
\maketitle
\onehalfspace

\begin{abstract}
The Riesz-Markov theorem identifies any positive, finite, and regular Borel measure on the complex unit circle with a positive linear functional on the continuous functions. By the Weierstrass approximation theorem, the continuous functions are obtained as the norm closure of the Disk Algebra and its conjugates.  Here, the Disk Algebra can be viewed as the unital norm-closed operator algebra of the shift operator on the Hardy Space, $H^2$ of the disk.  

Replacing square-summable Taylor series indexed by the non-negative integers, \emph{i.e.} $H^2$ of the disk, with square-summable power series indexed by the free (universal) monoid on $d$ generators, we show that the concepts of absolutely continuity and singularity of measures, Lebesgue Decomposition and related results have faithful extensions to the setting of `non-commutative measures' defined as positive linear functionals on a non-commutative multi-variable `Disk Algebra' and its conjugates.
\end{abstract}

\section{Introduction}

The results of this paper extend the Lebesgue decomposition of any finite, positive and regular Borel measure, with respect to Lebesgue measure on the complex unit circle, from one to several non-commuting (NC) variables. In \cite{JM-NCFatou}, we extended the concepts of absolute continuity and singularity of positive measures with respect to Lebesgue measure, the Lebesgue Decomposition, and the Radon-Nikodym formula of Fatou's Theorem to the non-commutative, multi-variable setting of `NC measures', \emph{i.e.} positive linear functionals on a certain operator system, the \emph{Free Disk System}. Here, the free disk system, $\A + \A ^*$, is the operator system of the \emph{Free Disk Algebra}, $\A := \mr{Alg} (I , L ) ^{-\| \cdot \| }$, the norm-closed operator algebra generated by the \emph{left free shifts} on the \emph{NC Hardy space}. (Equivalently, the left creation operators on the full Fock space over $\C ^d$.) We will recall in some detail below why this is the appropriate (and even canonical) extension of the concept of a positive measure on the circle to several non-commuting variables. The primary goal of this paper is to further develop the NC Lebesgue Decomposition Theory of an arbitrary (positive) NC measure with respect to NC Lebesgue measure (the `vacuum state' on the Fock space), by proving that our concepts of absolutely continuous (AC) and singular NC measures define positive hereditary cones, and hence that the Lebesgue Decomposition commutes with summation. That is, the Lebesgue Decomposition of the sum of any two NC measures is the sum of the Lebesgue Decompositions. (Here, we say a positive cone, $\mc{P} _0 \subset \mc{P}$ is \emph{hereditary} in a larger positive cone $\mc{P}$ if  $p_0 \in \mc{P} _0$, and $p_0 \geq p$ for any $p \in \mc{P}$ implies that $p \in \mc{P} _0$. The sets of absolutely continuous and singular positive, finite, regular Borel measures on the circle, $\partial \D$, with respect to another fixed positive measure, are clearly positive hereditary sub-cones.) In this paper we focus on positive NC measures and their Lebesgue decomposition with respect to NC Lebesgue measure. The study of more general complex NC measures, and the Lebesgue Decomposition of an arbitrary positive NC measure with respect to another will be the subject of future research. 

By the Riesz-Markov Theorem, any finite positive Borel measure, $\mu$, on $\partial \D$, can be identified with a positive linear functional, $\hat{\mu }$ on $\scr{C} ( \partial \D )$, the commutative $C^*-$algebra of continuous functions on the circle. By the Weierstrass Approximation Theorem,  $ \scr{C} (\partial \D ) = \left( \mc{A} (\D ) + \mc{A} (\D ) ^* \right) ^{-\| \cdot \| _\infty}, $  where $\mc{A} (\D )$ is the \emph{Disk Algebra}, the algebra of all analytic functions in the complex unit disk, $\D$, with continuous extensions to the boundary. In the above formula, elements of $\mc{A} (\D )$ are identified with their continuous boundary values and $\| \cdot \| _\infty$ denotes the supremum norm for continuous functions on the circle. The disk algebra can also be viewed as the norm-closed unital operator algebra generated by the shift, $S := M_z$, $\mc{A} (\D ) = \mr{Alg} (I, S) ^{-\| \cdot \| }$ (with equality of norms).  The shift is the isometry of multiplication by $z$ on the Hardy space, $H^2 (\D )$, and plays a central role in the theory of Hardy spaces. Here recall that the Hardy Space, $H^2 (\D )$, is the space of all analytic functions in $\D$ with square-summable MacLaurin series coefficients (and with the $\ell ^2$ inner product of these Taylor series coefficients at $0 \in \D$).  The positive linear functional $\hat{\mu}$ is then completely determined by the moments of the measure $\mu$: 
$$ \hat{\mu}  (S^k ) := \int _{\partial \D} \zeta ^k \mu  (d\zeta ). $$ 

The shift on $H^2 (\D )$ is isomorphic to the unilateral shift on $\ell ^2 (\N _0 )$, where $\N _0$, the non-negative integers, is the universal monoid on one generator.  A canonical several-variable extension of $\ell ^2 (\N _0 )$ is then $\ell ^2 (\F ^d )$, where $\F ^d$ is the free (and universal) monoid on $d$ generators, the set of all words in $d$ letters. One can define a natural $d-$tuple of isometries on $\ell ^2 (\F ^d)$, the \emph{left free shifts}, $L_k$, $1\leq k \leq d$ defined by
$L_k e_\alpha = e_{k \alpha}$ where $\alpha \in \F ^d$ and $\{ e_\alpha \}$ is the standard orthonormal basis. These left free shifts have pairwise orthogonal ranges so that the row operator $L := (L_1, ..., L_d ) : \ell ^2 (\F ^d ) \otimes \C ^d \rightarrow \ell ^2 (\F ^d)$ is an isometry from $d$ copies of $\ell ^2 (\F ^d )$ into one copy which we call the \emph{left free shift}. This Hilbert space of free square-summable sequences can also be identified with a `NC Hardy Space' of `non-commutative analytic functions' in a non-commutative open unit disk or ball of several matrix variables. Under this identification, the left free shifts become left multiplication by independent matrix-variables, see Section \ref{back}.  The immediate analogue of a positive measure in this non-commutative (NC) multi-variable setting is then a positive linear functional, or \emph{NC measure}, on the \emph{Free Disk System}:
$$ \left( \A  + \A  ^* \right) ^{-\| \cdot \| }, $$ where $\A := \mr{Alg} (I, L) ^{-\| \cdot \|}$ is the \emph{Free Disk Algebra}, the operator norm-closed unital operator algebra generated by the left free shifts. As in the classical theory, elements of the Free Disk Algebra can be identified with bounded (matrix-valued) analytic functions (in several non-commuting matrix variables) which extend continuously from the interior to the boundary of a certain non-commutative multi-variable open unit disk or ball. 

There is a fundamental connection between this work and the theory of row isometries (isometries from several copies of a Hilbert space into itself), or equivalently to the representation theory of the Cuntz-Toeplitz $C^*-$algebra, $\mc{E} _d = C^* (I , L )$ (the $C^*-$ algebra generated by the left free shifts), the universal $C^*-$algebra generated by a $d-$tuple of isometries with pairwise orthogonal ranges, and of the Cuntz $C^*-$algebra, $\mc{O} _d$, the universal $C^*-$algebra of an onto row isometry \cite{Cuntz}. Namely, applying the Gelfand-Naimark-Segal (GNS) construction to $(\mu , \A )$, where $\mu$ is any (positive) NC measure yields a GNS Hilbert space, $F^2 _d (\mu)$, and a $*-$representation $\pi _\mu$ of $\mc{E} _d$ so that $\Pi _\mu := \pi _\mu (L)$ is a row-isometry on $F^2 _d (\mu )$. A Lebesgue Decomposition for bounded linear functionals on the Free Disk Algebra, $\A$, has been developed by Davidson, Li and Pitts in the theory of Free Semigroup Algebras, \emph{i.e.} $WOT-$closed (weak operator topology closed) operator algebras generated by row isometries \cite{DLP-ncld,DKP-structure,DP-inv}. Building on this, Kennedy has constructed a Lebesgue Decomposition for row isometries \cite{MK-rowiso}, and we will explicitly work out the relationship between this theory and our Lebesgue Decomposition.

\subsection{Three approaches to classical Lebesgue Decomposition Theory}
\label{classmeas}

There are three approaches to classical Lebesgue Decomposition theory which will provide natural and equivalent extensions to NC measures. 

Let $\mu $ be an arbitrary finite, positive, and regular Borel measure on $\partial \D$, and as before, $m$ denotes normalized Lebesgue measure on the circle. As we will prove, one can construct the Lebesgue decomposition of  $\mu$ with respect to $m$ using reproducing kernel Hilbert space theory. Namely, setting $H^2 (\mu)$ to be the closure of the analytic polynomials in $L^2 (\mu , \partial \D )$, let 
$\scr{H} ^+ (H _\mu )$ be the space of all Cauchy Transforms of elements in $H^2 (\mu )$: If $h \in H^2 (\mu )$, 
$$ (\mc{C} _\mu h ) (z) := \int _{\partial \D } \frac{1}{1-z\ov{\zeta}} h (\zeta) \mu (d\zeta ). $$ Equipped with the inner product of $H^2 (\mu )$, this is a reproducing kernel Hilbert space of analytic functions in $\D$, the classical \emph{Herglotz Space}  with reproducing kernel: 
$$ K^\mu (z,w) = \frac{1}{2} \frac{H_\mu (z) + H_\mu (w) ^*}{1-zw^*} = \int _{\partial \D} \frac{1}{1-z\ov{\zeta}}\frac{1}{1-\zeta \ov{w}} \mu (d\zeta ), $$ and 
\ba H_\mu (z) &:= & \int _{\partial \D } \frac{1+z\ov{\zeta}}{1-z\ov{\zeta}} h (\zeta) \mu (d\zeta ) \nn \\
& = & 2 (\mc{C} _\mu 1 ) (z) - \mu (\partial \D ), \nn \ea is the Riesz-Herglotz integral transform of $\mu$, an analytic function with non-negative real part in $\D$ (see \cite[Chapter 1]{dB-SS}, or \cite[Chapter 1, Section 5]{dB-entire}). It is not hard to verify that domination of finite, positive, and regular Borel measures is equivalent to domination of the Herglotz kernels for their reproducing kernel Hilbert spaces of Cauchy Transforms:
$$ 0 \leq \mu \leq t^2 \la \quad \Leftrightarrow \quad K^\mu \leq t^2 K^\la; \quad \quad t>0. $$ 
Moreover, by a classical result of Aronszajn, domination of the reproducing kernels $K^\mu \leq t^2 K^\la$ is equivalent to bounded containment of the corresponding Herglotz spaces on $\D$, $\scr{H} ^+ (H_\mu ) \subseteq \scr{H} ^+ (H _\la)$, and the least such $t>0$ is the norm of the embedding map $\mr{e} _\mu : \scr{H} ^+ (H_\mu ) \hookrightarrow \scr{H} ^+ ( H_\la )$ \cite[Theorem I, Section 7]{Aron-RKHS}. Absolute continuity of measures on $\partial \D$ can also be recast in terms of containment of reproducing kernel Hilbert spaces. Namely, given two finite, positive, regular Borel measures $\la , \mu$, recall that $\mu$ is absolutely continuous with respect to $\la$ if there is a non-decreasing sequence of finite, positive, regular Borel measures $\mu _n$, which are each dominated by $\la$, and increase monotonically to $\mu$:
\begin{align*} 0 \leq \mu _n \leq \mu,  & \quad \quad  \mu _n \uparrow \mu, \\ 
 \mu _n \leq t_n ^2 \la, & \quad \quad t_n >0. \end{align*} 
Reproducing kernel Hilbert space theory then implies that each space of $\mu _n -$Cauchy transforms is contractively contained in the space of $\mu-$Cauchy transforms, and their linear span is dense in $\scr{H} ^+ ( H _\mu )$ since the $\mu _n$ increase to $\mu$,
$$ \bigvee \scr{H} ^+ (H _{\mu _n} ) = \scr{H} ^+ (H _\mu ). $$ Since each $\mu _n \leq t_n ^2 \la$, is dominated by $\la$, it also follows that each space of $\mu _n -$Cauchy transforms is boundedly contained in the space of $\la -$Cauchy transforms, and the intersection space:
$$ \mr{int} (\mu , \la ) := \scr{H} ^+ (H _\mu ) \bigcap \scr{H} ^+ (H _\la ), $$ is dense in the space of $\mu-$Cauchy transforms. In the case where $\la =m$ is normalized Lebesgue measure, one can check that $H_m \equiv 1$ is constant, so that $\scr{H} ^+ (H _m ) = H^2 (\D )$ is the classical Hardy space of the disk. It follows that one can take this as a starting point, and simply define a measure, $\mu$, to be absolutely continuous or singular (with respect to $m$) depending on whether the intersection space $$ \mr{int} (\mu , m ) := \scr{H} ^+ (H _\mu ) \bigcap H^2 (\D ), $$  is dense or trivial, respectively, in the space of $\mu-$Cauchy Transforms.  In this way one can develop Lebesgue Decomposition Theory using reproducing kernel techniques. It appears that this approach is new, even in this classical setting, and as shown in Corollary \ref{classLD}, this recovers the Lebesgue decomposition of any finite, positive and regular Borel measure on the unit circle with respect to normalized Lebesgue measure.

As discussed in the introduction, any positive, finite, regular Borel measure, $\mu$, on $\partial \D$, can be viewed as a positive linear functional, $\hat{\mu}$, on $\mc{A} (\D ) + \mc{A} (\D ) ^*$. Equivalently, $\mu$ (or $\hat{\mu}$) can be identified with the (generally unbounded) positive quadratic or sesquilinear form, 
$$ q_\mu (a_1 , a_2 ) := \int _{\partial \D } \ov{a_1 (\zeta )} a_2 (\zeta) \mu (d\zeta ); \quad \quad a_1, a_2 \in \mc{A} (\D ), $$ densely-defined in $H^2 (\D )$. Applying the theory of Lebesgue Decomposition of quadratic forms due to B. Simon yields:
$$ q_\mu = q_{ac} +q_s, $$ where $q_{ac}$ is the maximal positive form absolutely continuous to $q_m$, where $m$ is normalized Lebesgue measure, and $q_s$ is singular \cite{Simon1}. In this theory, a positive quadratic form with dense domain in a Hilbert space, $\H$, is said to be absolutely continuous if it is closable, \emph{i.e.} it has a closed extension. Here, a positive semi-definite quadratic form, $q$, is closed if its domain, $\dom{q}$ is complete in the norm $$ \| \cdot \| _{q+1} := \sqrt{ q (\cdot, \cdot ) + \ip{\cdot}{\cdot} _\H }.$$ Closed positive semi-definite forms obey an extension of the Riesz Representation Lemma: A positive semi-definite densely-defined quadratic form, $q$, is closed if and only if $q$ is the quadratic form of a closed, densely-defined, positive semi-definite operator, $T \geq 0$:
$$ q(h,g) = \ip{\sqrt{T}h}{\sqrt{T}g} _\H; \quad \quad h, g \in \dom{\sqrt{T}} = \dom{q}, $$
see \cite[Chapter VI, Theorem 2.21, Theorem 2.23]{Kato}.
If $q= q_\mu$, we will prove that $q_{ac} = q_{\mu _{ac}}$, and $q_s = q_{\mu _s}$ where 
$$ \mu = \mu _{ac} + \mu _s, $$ is the classical Lebesgue Decomposition of $\mu$ with respect to $m$, see Corollary \ref{classLD}. Indeed, if one instead defines $q_\mu$ as a quadratic form densely-defined in $L^2 (\partial \D )$, then it follows without difficulty in this case that $T$ is affiliated to $L^\infty (\partial \D )$ so that 
$$ q_\mu (f,g) = \int _{\partial \D } \ov{f(\zeta)} g(\zeta ) |h (\zeta ) | ^2 m (d\zeta ), $$ where $\sqrt{T} 1 = |h| \in L^2 (\partial \D )$. This shows that $|h| ^2 \in L^1 (\partial \D )$ is the Radon-Nikodym derivative of $\mu$ with respect to normalized Lebesgue measure, $m$. The Lebesgue Decomposition of quadratic forms in \cite{Simon1} is similar in this case to von Neumann's proof of the Lebesgue Decomposition theory \cite[Lemma 3.2.3]{vN3}. In \cite{JM-NCFatou}, we applied this quadratic form decomposition to the quadratic form, $q_\mu$, of any (positive) NC measure $\mu $ to construct an NC Lebesgue decomposition of $\mu$, $\mu = \mu _{ac} + \mu _s$ into absolutely continuous and singular NC measures $\mu _{ac}$ and  $\mu _s$, $0 \leq \mu _{ac} , \mu _s \leq \mu $, where $q_{\mu } = q_{\mu _{ac} } + q _{\mu _s}$ is the Lebesgue decomposition of the quadratic form $q_\mu$ \cite[Theorem 5.9]{JM-NCFatou}.

A third approach to Lebesgue Decompositon theory is to define a positive, finite, regular, Borel measure $\mu$, on $\partial \D$ to be absolutely continuous if the corresponding linear functional $\hat{\mu}$ on $\scr{C} (\partial \D ) = \left( \mc{A} (\D ) + \mc{A} (\D ) ^* \right) ^{-\| \cdot \| }$ has a weak$-*$ continuous extension to a linear functional on $\left( H^\infty (\D ) + H^\infty (\D ) ^* \right) ^{-wk-*} = L^\infty (\partial \D )$. 

This notion of absolute continuity for bounded linear functionals on $\A$ extends the classical notion of absolute continuity of a measure with respect to normalized Lebesgue measure on $\partial \D$, if one identifies finite positive Borel measures on $\partial \D$ with positive linear functionals on the classical Disk Algebra $\mc{A} _1 = \mc{A} (\D) \subset H^\infty (\D )$. Indeed in the case where $d=1$, $L^\infty _1 = H^\infty (\D)$, and  
$$ \left( H^\infty (\D ) + H^\infty (\D ) ^* \right) ^{-weak-*} \simeq L^\infty (\partial \D ),$$ a commutative von Neumann algebra. In this case, if $\hat{\mu} \in (\mc{A} (\D ) ^\dag ) _+ = \scr{C} (\partial \D ) ^\dag _+$ is any positive linear functional, the Riesz-Markov-Kakutani Theorem implies it is given by integration against a positive finite Borel measure, $\mu$, on $\partial \D$, and to say it has a weak$-*$ continuous extension to $(H^\infty (\D ) + H^\infty (\D ) ^* ) ^\dag _+ \simeq L^\infty (\partial \D ) ^\dag _+$ is equivalent to $\hat{\mu}$ being the restriction of a positive $\hat{\mu } \in L^\infty (\partial \D ) _*  \simeq L^1 (\partial \D )$. Equivalently, 
$$ \mu (d\zeta ) = \frac{\mu (d\zeta )}{m (d\zeta )} m (d\zeta ); \quad m-a.e., \quad \quad \frac{\mu (d\zeta )}{m (d\zeta )} \in L^1 (\partial \D ), $$ \emph{i.e.} $\mu$ is absolutely continuous with respect to Lebesgue measure. 

This definition of absolute continuity has an obvious generalization to the non-commutative setting of NC measures, \emph{i.e.} positive linear functionals on the Free Disk System, and this gives essentially the same definition of absolute continuity for linear functionals on the Free Disk Algebra introduced by Davidson-Li-Pitts \cite{DLP-ncld}. We will show that all three of these approaches extend naturally to the NC setting and yield the same Lebesgue Decomposition of any positive NC measure with respect to NC Lebesgue measure. 

\section{Background: The Free Hardy Space} \label{back}

We will use the same notation as in \cite{JM-NCFatou}, and we refer to \cite[Section 2]{JM-NCFatou} for a detailed introduction to the NC Hardy space and background theory.

The free monoid, $\F ^d$ is the set of all words in $d$ letters $\{ 1, ... , d \}$. This is the universal monoid on $d$ generators, with product given by concatenation of words, and unit $\emptyset$, the empty word containing no letters. The Hilbert space of square summable sequences indexed by $\F ^d$, $\ell ^2 (\F ^d )$, and the full Fock space over $\C ^d$,
$$ F^2 _d  := \bigoplus _{k=0} ^\infty \left( \C ^d \right) ^{k \cdot \otimes}  = \C \oplus \C ^d \oplus \left( \C ^d \otimes \C ^d \right) \oplus \left( \C ^d \otimes \C ^d \otimes \C^d \right) \oplus \cdots, $$ are naturally isomorphic.  This isomorphism is implemented by the unitary map $e _{i_1 \cdots i_k} \mapsto e_{i_1} \otimes \cdots \otimes e_{i_k}$, $i_k \in \{ 1, ... , d \}$, and $e_\emptyset \mapsto 1$ where $\{ e_{j} \}$ denotes the standard basis of $\C ^d$, and $1$ is the vacuum vector of the Fock space (which spans the subspace $\C \subset F^2 _d$). The free square-summable sequences, $\ell ^2 (\F ^d )$, can also be viewed as a Hilbert space of \emph{Free Non-commutative functions} on a \emph{non-commutative set} \cite{KVV,Pop-freeholo,SSS}. Namely, we can identify any $f \in \ell ^2 (\F ^d )$ with a formal power series in $d$ non-commuting variables $\mf{z} := (\mf{z} _1 , ... , \mf{z} _d )$, 
$$ f (\zeta ) := \sum _{\alpha \in \F ^d} \hat{f} _\alpha \mf{z} ^\alpha. $$ Here, if $\alpha = i_1 i_2 \cdots i_n$, $i _k \in \{ 1, ... ,d  \}$, we use the standard notation $\mf{z} ^\alpha = \mf{z} _{i_1} \mf{z} _{i_2} \cdots \mf{z} _{i_d}$. Foundational work of Popescu has shown that if $Z := (Z_1 , ... , Z_d ) : \H \otimes \C ^d \rightarrow \H$ is any strict (row) contraction on a Hilbert space, $\H$, then the above formal power series for $f$ converges absolutely in operator norm when evaluated at $Z$ (and uniformly on compacta) \cite{Pop-freeholo,SSS}. It follows that any $f \in \ell ^2 (\F ^d)$, can be viewed as a function in the Non-commutative (NC) open unit ball: 
$$ \B ^d _\N := \bigsqcup _{n=1} ^\infty \B ^d _n; \quad \B ^d _n := \left( \C ^{n\times n} \otimes \C ^d \right) _1, $$ where $\B ^d _n$ is the set of all strict row contractions on $\C ^n$. Moreover any such $f$ is a locally bounded \emph{Free Non-commutative function}, in the sense of \cite{Taylor2,KVV,Ag-Mc}. That is, it respects the grading, direct sums and similarities.  Any locally bounded free NC function (under mild, minimal assumptions on its NC domain) is automatically holomorphic, \emph{i.e.} it is both G\^{a}teaux and Fr\'{e}chet differentiable at any point $Z \in \B ^d _\N$ and has a convergent Taylor-type power series expansion about any point \cite[Chapter 7]{KVV}. It follows that we can identify $\ell ^2 (\F ^d)$ with the \emph{Free Hardy Space}:
$$ H^2 (\B ^d _\N ) := \left\{  f \in \mr{Hol} (\B ^d _\N )  \left| \ f(Z) = \sum _{\alpha \in \F ^d } \hat{f} _\alpha Z^\alpha, \ \sum |\hat{f} _\alpha | ^2 < \infty \right.  \right\}, $$ the Hilbert space of all (locally bounded hence holomorphic) NC functions in the NC unit ball $\B ^d _\N$ with square-summable Taylor-MacLaurin series coefficients. In the sequel, we will identify $F^2 _d$, $ \ell ^2 (\F ^d )$, and the free or \emph{NC Hardy space} $H^2 (\B ^d _\N )$, and use the terms Fock space and NC Hardy space interchangeably. 

As described in the introduction, the NC Hardy Space is equipped with a canonical left free shift, $L := M^L _{Z}$, the row isometry of left multiplication by the NC variables $Z = (Z_1, \cdots , Z_d ) \in \B ^d _\N$. Each component left free shift, $L_k$, $1 \leq k \leq d$ is an isometry on $H^2 (\B ^d _\N) $ and these have pairwise orthogonal ranges. Viewing the $L_k$ as isometries on $\ell ^2 (\F ^d)$, $L_k e_\alpha = e_{k \alpha}$, and the $L_k$ are also unitarily equivalent to the left creation operators on the Fock space, $F^2 _d$. One can also define isometric right multipliers, $R_k = M^R _{Z_k}$, the \emph{right free shifts} (which append letters to the right of words indexing the canonical orthonormal basis), and these are unitarily equivalent to the left free shifts via the transpose unitary on $\ell ^2 (\F ^d )$, $U_\dag$,
$$ U_\dag e_\alpha := e_{\alpha ^\dag}, $$ where if $\alpha = i_1 \cdots i_n \in \F ^d$, then
$ \alpha ^\dag := i_n \cdots i_1, $ its transpose. 

As in the single-variable setting, the Free Hardy Space $H^2 (\B ^d _\N )$ can be equivalently defined using (non-commutative) reproducing kernel theory \cite{BMV}. All NC-RKHS in this paper will be Hilbert spaces of free NC functions on the NC unit disk or ball, $\B ^d _\N$. Any Hilbert space, $\H$ of NC functions on $\B ^d _\N$, is a NC-RKHS if the linear point evaluation map, $K_Z ^* : \H  \rightarrow \left( \C ^{n\times n}, \mr{tr} _n \right)$ is bounded for any $Z \in \B ^d _n$. We will let $K_Z$, the \emph{NC kernel map}, denote the Hilbert space adjoint of $K_Z ^*$, and, for any $y, v \in \C ^n$, 
$$ K \{ Z , y , v \} := K_Z (yv^*) \in \H. $$ Furthermore, given $Z \in \B ^d _n, y, v \in \C ^n$ and $W \in \B ^d _m,  x, u \in \C ^m$ the linear map
$$ K(Z,W) [ \cdot ] : \C ^{n\times m } \rightarrow \C ^{n\times m}, $$ defined by 
$$ \ipcn{y}{K(Z,W)[vu^*]x} := \ip{K \{Z , y, v \} }{ K \{ W , x, u \} } _{\H}, $$ is completely bounded for any fixed $Z,W$ and completely positive if $Z=W$. This map is called the completely positive non-commutative (CPNC) kernel of $\H$. As in the classical theory there is a bijection between CPNC kernel functions on a given NC set and NC-RKHS on that set \cite[Theorem 3.1]{BMV}, and if $K$ is a given CPNC kernel on an NC set, we will use the notation $\H _{nc} (K)$ for the corresponding NC-RKHS of NC functions.  The NC Hardy space, $H^2 (\B ^d _\N )$, is then the non-commutative reproducing kernel Hilbert space (NC-RKHS) corresponding to the completely positive non-commutative (CPNC) Szeg\"{o} kernel on the NC unit ball, $\B ^d _\N$:
$$ K(Z,W) [ \cdot ] := \sum _{\alpha \in \F ^d} Z^\alpha [ \cdot ]  (W^\alpha) ^*; \quad H^2 (\B ^d _\N ) = \H _{nc} (K). $$ 

All NC-RKHS in this paper will consist of free holomorphic functions in the NC unit ball $\B ^d _\N$ so that any $f \in \H _{nc} (K)$, has a convergent Taylor-MacLaurin series at $0 \in \B ^d _1$, 
$$ f(Z) = \sum _{\alpha \in \F ^d } Z^\alpha \hat{f} _\alpha ; \quad \quad Z \in \B ^d _n, $$ 
and the linear coefficient evaluation functionals: 
$$  f \stackrel{\ell _\alpha}{\rightarrow} \hat{f} _\alpha; \quad \quad \alpha \in \F ^d, $$ are all bounded. We will let $K_\alpha $ denote the \emph{coefficient evaluation vector}:
$$ \ip{K_\alpha}{f}_{\H _{nc} (K)} = \ell _\alpha (f) = \hat{f} _\alpha, \quad \quad \alpha \in \F ^d, $$ and we will typically write $\ell _\alpha =: K_\alpha ^*$. If $K$ is the NC-Szeg\"{o} kernel of the Free Hardy Space, then
$$ K_\alpha (Z) = Z^\alpha, $$ \emph{i.e.} $K_\alpha$ can be identified with the free monomial $L^\alpha 1 \in F^2 _d$.

If $\H _{nc} (K)$ is an NC-RKHS of NC functions on $\B ^d _\N$, NC functions $F,G$ on $\B ^d _\N$ are said to be left or right NC multipliers, respectively, if for any $f \in \H _{nc} (K)$, $F \cdot f$, or $f \cdot G$ belong to $\H _{nc} (K)$. As in the classical theory any left or right multiplier defines a bounded linear operator on $\H _{nc} (K)$, 
$$ (M^L _F f) (Z) := F(Z) f(Z), \quad \quad (M^R _G f ) (Z) := f(Z) G(Z), $$ and under this identification the left and right multiplier algebras of $\H _{nc} (K)$ are unital and closed in the weak operator topology (WOT).  These NC multiplier algebras are denoted by $\mr{Mult} _L (\H _{nc} (K) )$ or $\mr{Mult} _R (\H _{nc} (K) )$, respectively. The left multiplier algebra of the Free Hardy Space provides a non-commutative generalization of $H^\infty (\D) = \mr{Mult} (H^2 (\D ) )$:
$$ H^\infty (\B ^d _\N ) :=   \left\{ f \in \mr{Hol} (\B ^d _\N )  \left| \ \sup _{Z \in \B ^d _\N } \| f (Z ) \| < \infty \right. \right\} = \mr{Mult} _L ( H^2 (\B ^d _\N ) ).$$  This left multiplier algebra can also be identified with 
$$ L^\infty _d := \mr{Alg} (I, L_1, ... , L_d ) ^{-weak-*} = \mr{Alg} (I , L_1, ... , L_d ) ^{-WOT}, $$ the (left) \emph{analytic Toeplitz algebra}.  Here, note that the weak operator (WOT) and weak$-*$ topologies coincide on $L^\infty _d$, \cite[Corollary 2.12]{DP-inv}.  Here, and throughout, we write $\A + \A ^*$ in place of $(\A + \A ^* ) ^{- \| \cdot \| }$ to simplify notation. We also define $R^\infty _d= \mr{Alg} (I, R_1, ..., R_d ) ^{-WOT}$, the right free analytic Toeplitz algebra, and $R^\infty _d = U_\dag (L^\infty _d ) U_\dag$ is the image of $L^\infty _d $ under adjunction by the transpose unitary of $F^2 _d$. As in \cite{DP-inv,Pop-entropy} a left (or right) free multiplier of the Free Hardy Space will be called \emph{inner} if the corresponding (left or right) multiplication operator is an isometry, and \emph{outer} if the corresponding (left or right) multipication operator has dense range.

\section{Non-commutative measures} \label{NCmeasuresect}

\begin{defn} \label{NCmeasure}
Let $(\A ^\dag ) _+$ denote the set of all positive linear functionals on the (norm-closure of the) operator system $\A + \A ^*$, the \emph{Free Disk System}.  We will call such a functional a non-commutative or NC measure.
\end{defn}
 
\begin{defn} \label{rightHerglotzdefn}
A free holomorphic function, $H$ in $\B _\N ^d$ is a (right) NC Herglotz function if and only if the NC kernel:
$$ K^H (Z,W) := \frac{1}{2} K(Z,W)\left[ H(Z) (\cdot) + (\cdot) H(W) ^* \right] \geq 0,$$ is a CPNC kernel on $\B ^d _\N$, where $K(Z,W)$ is the free Szeg\"{o} kernel.
\end{defn}
As in the classical setting, there is a natural bijection between NC Herglotz functions and NC measures. Given any NC measure $\mu \in (\A ^\dag ) _+$, its moments define an NC Herglotz function:
$$ H_\mu (Z) := \mu (I) + \frac{1}{2} \sum _{\alpha \neq \emptyset } Z^\alpha \mu (L^\alpha ) ^*.$$ Conversely, any NC Herglotz function has the MacLaurin series expansion,
$$ H_\mu (Z) := H_\emptyset + \sum _{\alpha \neq \emptyset} Z^\alpha H_\alpha, $$ and setting
$$ \mu _H (I) = \re{H_\emptyset}, \quad \mbox{and} \quad  \mu _H (L^\alpha ) := 2 H_\alpha ^*, $$ defines a (positive) NC measure \cite{JM-freeCE}. (This Taylor-Maclaurin series converges absolutely in $\B ^d _\N$, and uniformly on $r \B ^d _\N$ for any $0<r<1$.)

Given any NC Herglotz function, $H$, the corresponding NC-RKHS $\H _{nc} (K^H)$ is then a Hilbert space of NC holomorphic functions in $\B ^d _\N$ by NC-RKHS theory \cite{BMV}. If $\mu \in (\A ^\dag ) _+$ is the unique NC measure corresponding to $H$, we will usually write $K ^H = K ^\mu$, and we will use the notation $\scr{H} ^+ (H_\mu ) := \H _{nc} (K ^\mu )$ for the right Free Herglotz Space of $H_\mu$. Here, we will also write $H = H_\mu$ (or sometimes $\mu = \mu _H$). As described in \cite{JM-freeCE,JM-freeAC}, if $H = H _\mu$, there is a natural onto isometry, the (right) \emph{Free Cauchy Transform}, $\mc{C} _\mu : F^2 _d (\mu ) \rightarrow \scr{H} ^+ (H_\mu )$, defined as follows: For any free polynomial $p \in \C \{ L_1 , ... , L_d \} \subseteq  F^2 _d (\mu )$
\ba (\mc{C} _\mu p ) (Z) & = & (\mr{id} _n \otimes \mu  ) \left( (I_{n\times F^2} -Z(\dag \circ L)^*) ^{-1} (I_n \otimes p(L)) \right) \nn \\
& := & \sum _{\alpha \in \F ^d } Z^\alpha \mu \left( (L^\alpha ) ^* p(L)  \right) \nn \\
& = &  \sum _{\alpha \in \F ^d } Z^\alpha \ip{L^\alpha}{p(\Pi _\mu ) ( I + N _\mu )} _\mu. \label{FreeCT}    \ea
In the above, for any $Z \in \B ^d _\N$, $ZL^* := Z_1 \otimes L_1 ^* + ... + Z_d \otimes L_d ^*$ is a strict contraction. The map $\mr{id} _n$ is the identity map on $\C ^{n\times n}$, and $I_{n \times F^2} := I_n \otimes I _{F^2 _d}$. The final formula above extends to arbitrary $x \in F^2 _d (\mu)$.   (In the first line of the formula above, the $\dag$ symbol means that one needs to take the transpose of all words in the geometric sum of $(I _{n \times F^2} - ZL^* ) ^{-1}$ to obtain the second line.)

\subsection{Non-commutative Lebesgue measure} 
\label{NCLebesgue}
Classically, the Riesz-Herglotz transform, $H_m (z)$, of normalized Lebesgue measure, $m$ on $\partial \D$ is the constant function $H_m \equiv 1$. It is then natural to expect that in the NC multi-variable theory, the role of normalized Lebesgue measure should be played by the unique NC measure corresponding to the constant NC Herglotz function: 
$$ H (Z) := I _n; \quad \quad Z \in \B ^d _n. $$ It is easy to check that the unique NC measure (which we also denote by $m$), $m = \mu _H$, corresponding to the NC function $H(Z) = I_n $ is the vacuum state on the Fock space:
$$ m (L^\alpha ) := \ip{1}{L^\alpha 1} _{F^2} = \delta _{\alpha, \emptyset}. $$

\begin{defn}
    The vacuum state $m \in (\A ^\dag ) _+$ will be called (normalized) \emph{NC Lebesgue measure}.
\end{defn}

\subsection{Left regular represenations of the Cuntz-Toeplitz algebra}

If $\mu \in (\A ^\dag ) _+$, the Gelfand-Naimark-Segal (GNS) space $F^2 _d (\mu )$ is the the Hilbert space completion of $\A$ modulo zero length vectors with respect to this pre-inner product:
$$ \ip{a_1}{a_2} _\mu := \mu (a_1 ^* a_2); \quad \quad a_1, a_2 \in \A. $$ We will typically write $a + N_\mu$ for the equivalence class of $a$ in $F^2 _d (\mu )$, where $N_\mu \subseteq \A$ is the left ideal of all elements of zero length.  Moreover, the left regular representation: $\pi _\mu : \A  \rightarrow \mc{L} (F^2 _d (\mu  ) )$, 
$$ \pi _\mu (L^\alpha) (a + N_\mu) := L^\alpha a + N _\mu, $$ is completely isometric and extends uniquely to a $*-$representation of the Cuntz-Toeplitz algebra $\mc{E} _d = C^* (I, L)$ on $\mc{L} (F^2 _d (\mu ) ).$ In particular, 
$$ \Pi _\mu = \pi _\mu (L) := (\pi _\mu (L_1) , ... , \pi _\mu (L_d) ) : F^2 _d (\mu ) \otimes \C ^d \rightarrow F^2 _d (\mu ), $$ is a (row) isometry, and we write $(\Pi _\mu) _k := \pi _\mu (L_k).$  Again, if $d=1$ then 
$$ F^2 _1 (\hat{\mu } ) \simeq H^2 (\mu ), \quad \mbox{and} \quad \Pi _{\hat{\mu}} \simeq M_\zeta | _{H^2 (\mu)}, $$ where $\hat{\mu}$ is, as before, the positive linear functional corresponding to the positive measure, $\mu$.

\begin{remark}
It is not difficult to see that one can obtain (up to unitary equivalence) any cyclic row isometry with the above construction, \emph{i.e.} any cyclic row isometry is the left regular GNS representation coming from an NC  measure. 

More generally one can construct any $*-$representation of the Cuntz-Toeplitz algebra (up to unitary equivalence), by considering Stinespring-GNS representations of operator-valued NC measures, \emph{i.e.} completely positive operator-valued maps on the free disk system.
\end{remark}

\subsection{Image of GNS row isometry under Free Cauchy Transform} \label{CTofGNS}

The image of the GNS row isometry $\Pi _\mu$ under the Free Cauchy Transform is a natural isometry on the Free Herglotz space:
\be V _\mu := \mc{C} _\mu \Pi _\mu (\mc{C}  _\mu ) ^*. \ee 

The range $\mc{R}$ of the row isometry $V _\mu $ is:
\be \mc{R} :=\bigvee \left( K^{H_\mu} \{ Z , y , v \} - K^{H _\mu} \{ 0 _n , y, v \} \right) = \bigvee _{\alpha \neq \emptyset} K^{H \mu } _\alpha, \ee and
for any $Z \in \B ^d _n, \ v, y \in \C^n$,
\be 
V_\mu   ^* \left( K^{H_\mu} \{ Z , y , v \} - K^{H_\mu} \{ 0 _n , y, v \} \right) = K ^{H_\mu} \{ Z , Z ^* y ,  v \}  \ee
(so that the span of all such vectors is dense in $\scr{H} ^+  (H_\mu) \otimes \C ^d$).

The image of $\ran{V_\mu}$ under $(\mc{C} _\mu ) ^*$ is:
\be \label{nconmon} F^2 _d (\mu ) _0 = \bigvee _{\alpha \neq \emptyset} L^\alpha + N _\mu, \ee  the closed linear span of the non-constant free monomials in $F^2 _d (\mu )$. If $F \in \scr{H} ^+ (H_\mu)$ is orthogonal to $\ran{V_\mu}$, then $F$ is a constant NC function: For any $Z \in \B ^d _n$, 
$$ F(Z) = I_n F(0)  $$ \emph{i.e.} $F \equiv F(0) \in \C$ is constant-valued. See \cite[Section 4.4]{JM-freeCE} for details.

\begin{remark}
Recall that if $\mu = m$ is normalized NC Lebesgue measure (the vacuum state), then $H_\mu (Z) = I_n $ for any $Z \in \B ^d _n$ so that the NC Herglotz kernel, $K^{H_m} = K$ reduces to the NC Szeg\"{o} kernel and $\scr{H} ^+ (H_m) = H^2 (\B ^d _\N ) $ is simply the Free Hardy Space. In this case $V_m = M^L _{Z} \simeq L$ is the left free shift. 
\end{remark}

\section{Cauchy Transforms of NC measures}

\label{FreeCTsect}

The goal of this section is to define absolutely continuous and singular NC measures, and to show that any positive NC measure $\mu \in (\A ^\dag) _+$ has a unique Lebesgue Decomposition, $\mu = \mu _{ac} + \mu _s$, into absolutely continuous and singular parts, $\mu _{ac}, \mu _s \in (\A ^\dag ) _+$. 

Recall, as discussed in Section \ref{classmeas}, that domination and absolute continuity of any finite, positive, regular Borel measure, $\mu$, on $\partial \D$, can be described in terms of the intersection of the RKHS of $\mu-$Cauchy transforms with the Hardy space, $H^2 (\D )$. In particular, domination of measures is equivalent to domination of the reproducing kernels for their spaces of Cauchy transforms so that the following NC analogue of a reproducing kernel theory result due to Aronszajn applies, see \cite[Theorem 5.1]{Paulsen-rkhs} \cite[Theorem I, Section 7]{Aron-RKHS}: 

\begin{thm} \label{NCdom}
Let $K_1, K_2$ be CPNC kernels on an NC set $\Om := \bigsqcup \Om _n$, where 
$$ \Om _n := \Om \bigcap \left( \C ^{n \times n} \otimes \C ^d \right). $$ Then $K_1 \leq t^2 K_2$ for some $t>0$ if and only if  
$$ \mc{H} _{nc} (K_1 ) \subseteq \mc{H} _{nc} (K_2), $$ and the norm of the 
embedding $\mr{e} : \mc{H} _{nc} (K_1 ) \hookrightarrow \mc{H} _{nc} (K_2)$ is at most $t$.
\end{thm}

Moreover, as in the single-variable setting, it is easy to verify that domination of (positive) NC measures $\mu, \la \in (\A ^\dag )_+$ is equivalent to domination of the NC kernels for their spaces of Cauchy Transforms:  If $\mu , \la \in (\A ^\dag ) _+$ are positive NC measures and $\mu$ is dominated by $\la$, \emph{i.e.} there is a $t>0$ so that $\mu \leq t^2 \la$, then there is a linear embedding, $E_\mu : F^2 _d (\la ) \hookrightarrow F^2 _d (\mu )$ defined by 
$$ E_\mu ( p(L) + N_\la ) = p(L) + N_\mu, \quad \quad p \in \fp, $$ with norm at most $t$.

\begin{lemma}{ (\cite[Lemma 5.3]{JM-NCFatou})} \label{embeddings}
Given $\mu, \la \in (\A ^\dag ) _+$, there is a $t>0$ so that $K^\mu \leq t^2 K^\la$ if and only if $\mu \leq t^2 \la$. If $\mu \leq t^2 \la$, then the linear embeddings $\mr{e} _\mu : \scr{H} ^+ (H _\mu ) \hookrightarrow \scr{H} ^+ (H _\la )$ and 
$E_\mu : F^2 _d (\la ) \hookrightarrow F^2 _d (\mu)$ have norm at most $t>0$ and are related by:
$$ E_\mu = \mc{C} _\mu ^* \mr{e} _\mu ^* \mc{C} _\la. $$
\end{lemma}

Motivated by the discussion of Section \ref{classmeas}, we define: 
\begin{defn} \label{ncldDef}
A positive NC measure $\mu \in (\A ^\dag ) _+$ is \emph{absolutely continuous} (AC) (with respect to NC Lebesgue measure, $m$) if the intersection of its space of Cauchy Transforms, $\scr{H} ^+ (H _\mu )$, with the Free Hardy Space is dense: 
$$ \scr{H} ^+ ( H _\mu ) =  \left( \scr{H} ^+ (H _\mu ) \bigcap H^2 (\B ^d _\N ) \right) ^{-\| \cdot \| _{H_\mu}}. $$ 
The NC measure $\mu$ is \emph{singular} (again with respect to NC Lebesgue measure) if 
$$ \mr{int} (\mu , m) := \scr{H} ^+ (H _\mu ) \bigcap H^2 (\B ^d _\N ) =\{ 0 \}. $$

The sets of all absolutely continuous and singular positive NC measures will be denoted by 
$AC (\A ^\dag ) _+$ and $\mr{Sing} (\A ^\dag ) _+$, respectively.
\end{defn}
Here, recall that the space of all Cauchy transforms with respect to NC Lebesgue measure, $m$, $\scr{H} ^+ (H _m )$ is the NC Hardy Space $H^2 (\B ^d _\N )$. Corollary \ref{classLD} will show that this definition recovers the classical Lebesgue decomposition of any finite, positive and regular Borel measure on the circle with respect to Lebesgue measure, in the single-variable setting.

Our goal now is to decompose any positive NC measure, $\mu \in (\A ^\dag ) _+ $ into absolutely continuous and singular parts using reproducing kernel theory by considering the intersection of the space of NC $\mu-$Cauchy transforms with the NC Hardy space. For any (positive) NC measures $\mu,\la$, one has that
$H_{\mu +\la} = H_\mu + H_\la$, and it follows that the NC Herglotz kernel of the NC measure $\ga := \mu +\la$ obeys:
$$ K^{\ga } (Z,W) = K^\mu (Z,W) + K ^\la (Z,W). $$ In particular, one can prove the following NC analogue of a result on sums of reproducing kernels due to Aronszajn (applied to the special case of NC Herglotz Spaces), \cite[Section 6]{Aron-RKHS}, \cite[Theorem 5.7]{Paulsen-rkhs}:

\begin{thm} \label{NCsum}
If $\mu, \la \in (\A ^\dag ) _+$ then $\scr{H} ^+ (H_{\mu +\la} ) = \scr{H} ^+ (H_\mu ) + \scr{H} ^+ (H _\la )$ and the NC reproducing kernel of $\scr{H} ^+ (H _{\mu +\la })$ is $K^{\mu +\la} (Z,W) = K^\mu (Z,W) + K ^\la (Z,W)$. The norm of any $h \in \scr{H} ^+ (H _{\mu +\la} )$ is:
$$ \| h \| ^2 _{H_{\mu +\la}} = \mr{min} \left\{ \left. \| h_1 \| ^2 _{H_\mu} + \| h_2 \| ^2 _{H_\la} \right| h_1 \in \scr{H} ^+ ( H_\mu ), \ h_2 \in \scr{H} ^+ (H_\la ), \ \mbox{and} \  h = h_1 + h_2 \right\}. $$
In particular, $$\scr{H} ^+ ( H_{\mu +\la} ) \simeq \scr{H} ^+ ( H_\mu ) \oplus \scr{H} ^+ (H_\la )$$ if and only if the intersection space: $$ \mr{int} (\mu, \la ) := \scr{H} ^+ ( H_\mu ) \bigcap \scr{H} ^+ ( H_\la ) = \{ 0 \}, \quad \mbox{is trivial.} $$
\end{thm}
Applying the inverse free Cauchy Transform, one has $\scr{H} ^+ ( H _{\mu +\la } ) \simeq \scr{H} ^+ (H _\mu ) \oplus \scr{H} ^+ (H_\la )$ if and only if 
$$ F^2 _d (\mu +\la) \simeq F^2 _d (\mu ) \oplus F^2 _d (\la ). $$ 
\begin{proof}
The proof is similar to the classical RKHS result, see \cite[Theorem 5.7]{Paulsen-rkhs}. Since $H_{\mu +\la} = H_\mu + H_\la$, it follows as in the classical theory that 
$K^{\mu +\la} (Z,W ) = K^\mu (Z,W) + K ^\la (Z,W), $ that $\scr{H} ^+ (H _{\mu +\la} ) = \scr{H} ^+ (H_\mu) + \scr{H} ^+ (H _\la ), $ and that the map $W$ from $\scr{H} ^+ (H_{\mu +\la} )$ into the direct sum $\scr{H} ^+ (H _\mu ) \oplus \scr{H} ^+ (H _\la )$ defined by
$$ W K_Z ^{\mu +\la} := K^\mu _Z \oplus K^\la _Z, $$ is an isometry onto the subspace
$$ S := \bigvee K^\mu \{ Z , y , v \} \oplus K^\la \{ Z , y , v \}, $$ with orthogonal complement 
$$ S^\perp = \{ f \oplus -f | \ f \in \scr{H} ^+ (H _\mu ) \bigcap \scr{H} ^+ (H _\la ) \}. $$
In particular, one has the direct sum decomposition if and only if the intersection space is trivial.
\end{proof}

\begin{thm} \label{NCintersect}
Given any two (positive) NC measures $\mu, \la \in (\A ^\dag ) _+$, the intersection space
$$ \mr{int} (\mu , \la ) := \scr{H} ^+ ( H _\mu ) \bigcap \scr{H} ^+ (H _\la ), $$ is 
both $V_\mu$ and $V_\la$ co-invariant, and 
$$ V_\mu ^* | _{\mr{int} (\mu , \la )} = V_\la ^* | _{\mr{int} (\mu , \la )}. $$
\end{thm}
\begin{lemma}
Let $\mbf{h} \in \mr{Hol} (\B ^d _\N ) \otimes \C ^d$. Then $Z \mbf{h} (Z) = 0_n$ for all $Z \in \B ^d _n$ implies that $\mbf{h} \equiv 0$.
\end{lemma}
\begin{proof}
This follows from basic NC analytic function theory. Let $g(Z) = Z \mbf{h} (Z) \in \mr{Hol} (\B ^d _\N )$, so that $g \equiv 0$. Any $g \in \mr{Hol} (\B ^d _n )$ has the Taylor-Taylor series expansion about $0_n$:
$$ g(Z)  = \sum _{k=0} ^\infty \frac{1}{k!} (\partial _{Z} ^k g) (0_n), $$ where 
$$ (\partial _{Z} g) (W) := \left. \frac{d}{dt} g(W + t Z ) \right| _{t=0}, $$ is the G\^{a}teaux derivative of $g$ at $W$ in the direction of $Z$, and the $\partial _Z ^k$ are the higher order G\^{a}teaux derivatives. This is a homogeneous polynomial decomposition, setting 
$$ g ^{(k)} (Z) := (\partial _Z ^k g ) (0_n), $$ each $g ^{(k)} (Z)$ is a homogeneous free polynomial of degree $k$. It follows that if 
$$ \mbf{h} = \bpm h_1 \\ \vdots \\ h_d \epm, $$ and each $h_j (Z)$ is the sum of homogeneous polynomials $h_j  ^{(k)} (Z)$, then, 
$$ g^{(k)} (Z) = Z_1 h_1 ^{(k-1)} (Z) + \cdots + Z_d h_d ^{(k-1)} (Z) ; \quad \quad k \geq 1. $$
Since $g$ vanishes identically, so do all of the $g^{(k)} (Z) = (\partial ^k _Z g ) (0_n ) $, for $k \geq 0$. It further follows that each of the $h_j ^{(k)}$ vanish identically. Indeed, one easy way to see this is that each $h_j ^{(k)}$ is a homogeneous free polynomial in the Fock Space $F^2 _d$, and 
$$ g^{(k)} (Z) = (L \mbf{h} ^{(k)} ) (Z); \quad \quad \mbf{h} ^{(k)} (Z) := \bpm h^{(k)} _1 \\ \vdots \\ h^{(k)} _d \epm. $$ It follows that each $\mbf{h} ^{(k)}$ is in the kernel of the left free shift. Since the left free shift is an isometry, each $h ^{(k)} _j \equiv 0$ vanishes identically for $1\leq j \leq d$.
\end{proof}
\begin{proof}{ (of Theorem \ref{NCintersect})}
If $f \in \scr{H} ^+ (H _\mu ) \bigcap \scr{H} ^+ (H _\la )$ then observe that
\ba Z (V _\mu ^* f ) (Z) & = & (V _\mu K_Z ^\mu Z^* ) f \nn \\
& = & (K_Z ^\mu - K_{0_n} ^\mu ) ^* f = f(Z) - f(0_n ) \nn \\
& = & Z (V_\la ^* f ) (Z). \nn \ea By the previous lemma it follows that
$$ (V_{\mu , k} ^* f ) (Z) = (V_{\la, k} ^* f ) (Z); \quad \quad 1\leq k \leq d, $$ agree
so that $V_{\la, k} ^* f = V _{\mu , k } ^* f \in \scr{H} ^+ (H _\mu ) \cap \scr{H} ^+ (H _\la )$ for each $1 \leq k \leq d$, and the intersection space is both $V_\mu$ and $V_\la-$co-invariant.
\end{proof}

\begin{thm}\label{redintersect}
If $\mathcal M$ is a closed subspace of $\scr{H}^+(H_\mu)$ which is reducing for $V_\mu$, then there exists an NC measure $\gamma\leq \mu$ such that 
\[
\mathcal M = \scr H^+(H_\gamma).
\]
\end{thm}
\begin{proof} It is easier to work in the $F^2_d (\mu )$ model, the conclusions then carry over to $\scr{H}^+(H_\mu)$ via the NC Cauchy transform.  If $\mathcal M\subset F^2_d (\mu )$ is any reducing subspace for $\Pi_\mu$, letting $P$ be the orthogonal projection on $\mathcal M$, we can define a new NC measure $\gamma$ by the formula
\[
\gamma( L^\alpha )= \ip{I+N_\mu }{P \Pi _\mu ^\alpha (I + N_\mu ) } _\mu = \langle I + N_\mu , P (L^\alpha + N _\mu)  \rangle_\mu , \quad \alpha \in \F ^d.
\]
We extend $\ga$ in the natural way to a linear functional on the free disk system by $\ga ( (L^\alpha ) ^* ) := \ga (L^\alpha ) ^*$. It remains to check that $\ga$ is a positive linear functional. By \cite[Lemma 4.6]{JM-freeAC}, any positive element in the free disk system is the norm-limit of sums of squares of free polynomials, so that it suffices to check that $\ga (p(L)^* p(L) ) \geq 0$ for any $p \in \fp$. Given any $p \in \fp$, let $u \in \fp$ be such that $p(L) ^* p(L) = u(L) + u(L) ^*$. Using that the orthgonal projection, $P$, commutes with the GNS representation $\Pi _\mu$, it is then not difficult to verify that
\ba \ga (p(L)^* p(L) ) & = & \ga (u (L) ) ^* + \ga (u (L) ) \nn \\
& = & \ip{p(L) + N_\mu}{P (p(L) + N_\mu)}_\mu \geq 0, \nn \ea so that $\ga \in (\A ^\dag ) _+$.
It is then evident $\mathcal M$ is isometrically identified with $F^2_d (\gamma )$ and that the image of $\mathcal M\subset F^2_d (\mu )$ under the Cauchy transform is equal to $\scr H^+(H_\gamma)$. In particular, $\ga \leq \mu$.
\end{proof}

\begin{prop} \label{reduceprop}
Given $\la , \mu \in (\A ^\dag )_+$, if $\scr{H} ^+ (H _\la)$ contains the constant NC functions, then $\scr{H} ^+ (H _\mu ) \bigcap \scr{H} ^+ (H _\la)$ is reducing for $V_\mu$.
\end{prop}
Clearly this applies to $\la =m$ since $H^2 (\B ^d _\N ) = \scr{H} ^+ (H _m )$ contains the constant NC functions.
\begin{proof}
Theorem \ref{NCintersect} shows that this intersection space is co-invariant for $V_\mu$. Conversely, given $f \in \scr{H} ^+ (H _\mu ) \bigcap \scr{H} ^+ (H _\la )$, observe that
\ba (V _{\mu, k} f ) (Z) - (V _{\mu , k} f ) (0_n ) & = & Z_k f(Z) \nn \\
& = & (V _{\la , k} f ) (Z) - (V _{\la, k} f ) (0_n ), \nn \ea so that 
$$ (V _{\mu , k} f ) (Z) = (V _{\la ,k} f ) (Z) + c I_n, $$ where $c := (V_{\mu, k} f ) (0) - (V_{\la, k} f) (0)$ is constant. Since $\scr{H} ^+ (H_\la)$ contains all the constant NC functions, it follows that $$ V_{\mu , k} f \in \scr{H} ^+ (H _\la) \bigcap \scr{H} ^+ (H _\mu ) $$  also belongs to the intersection space.
\end{proof}

\begin{remark} \label{NC-CTfail}
The conclusion of Proposition~\ref{reduceprop} can fail without the hypothesis about the constant functions, indeed this failure occurs even in one variable---this will be the case if we take, for example, $\mu$ and $\lambda$ to be the restriction of Lebesgue measure to the upper and lower semicircle respectively.  Moreover, in this example $\mu$ and $\lambda$ are mutually singular, but it is easy to check that the intersection space $\scr{H}^+(\mu)\cap \scr{H}^+(\lambda)$ is non-trivial, so our current Cauchy transform method is in general inadequate for the problem of computing the Lebesgue decomposition for arbitrary pairs of measures on the circle. 
\end{remark}

\begin{thm} \label{ncld2thm}
Any positive NC measure $\mu \in (\A ^\dag ) _+$ has the Lebesgue Decomposition $\mu = \mu _{ac} + \mu _s$, where $ 0 \leq \mu _{ac} , \mu _s \leq \mu$ are the (positive) absolutely continuous and singular NC measures defined by
$$ \scr{H} ^+ (H _{\mu _{ac}} ) := \left( \scr{H} ^+ (H _\mu ) \bigcap H^2 (\B ^d _\N ) \right) ^{-\| \cdot \| _{H_\mu}}, $$   and
$$ \scr{H} ^+ (H _{\mu _s} ) := \scr{H} ^+ (H _\mu ) \ominus \scr{H} ^+ (H _{\mu _{ac}} ). $$
Both $\scr{H} ^+ (H _{\mu _{ac}} )$ and $\scr{H} ^+ (H _{\mu _s})$ are reducing for $V_\mu$ and
$$ \scr{H} ^+ (H _\mu ) = \scr{H} ^+ (H _{\mu _{ac}} ) \oplus \scr{H} ^+ (H _{\mu _s} ). $$
\end{thm}
The direct sum decomposition of this theorem implies, by inverse Cauchy Transform, that
$$ F^2 _d (\mu ) = F^2 _d (\mu _{ac} ) \oplus F^2 _d (\mu _s ), $$ and these orthogonal subspaces are both reducing for $\Pi _\mu$.  
\begin{proof}
This is an immediate consequence of Theorem \ref{redintersect} and Proposition \ref{reduceprop}.
\end{proof}

\begin{thm} \label{ACcone}
The set $AC (\A ^\dag ) _+$ is a positive cone.
\end{thm}
\begin{proof}
Suppose that $\la , \mu \in AC (\A ^\dag ) _+$ and let $\ga = \la +\mu$. Then by Theorem \ref{NCintersect}, 
$$ \scr{H} ^+ (H _\ga ) = \scr{H} ^+ (H _\mu ) + \scr{H} ^+ (H _\la ), $$ and both $\scr{H} ^+ (H_\la ), \scr{H} ^+ (H _\mu )$ are contractively contained in $\scr{H} ^+ (H _\ga )$ by Theorem \ref{NCdom}, so that any $h \in \scr{H} ^+ (H _\ga )$ can be decomposed as $h = f +g$ for $f \in \scr{H} ^+ (H _\mu )$ and $g \in \scr{H} ^+ (H _\la )$. Since both $\la, \mu$ are AC, there is a $H_\mu -$norm convergent sequence $(f_n ) \subset \scr{H} ^+ (H _\mu ) \bigcap H^2 (\B ^d _\N )$ so that $f_n \rightarrow f$ in $\scr{H} ^+ (H_\mu )$. Similarly  there is a sequence $(g_n) \subset \scr{H} ^+ (H _\la ) \bigcap H^2 (\B ^d _\N )$ so that $g_n \rightarrow g$ in $\scr{H} ^+ (H _\la )$. Let $\mr{e} _\mu , \mr{e} _\la$ be the contractive embeddings of $\scr{H} ^+ (H _\mu ), \scr{H} ^+ (H _\la )$ into $\scr{H} ^+ (H _\ga )$. The sequence,
$$ h_n := \mr{e}_\mu f_n + \mr{e}_\la g_n \in \scr{H} ^+ (H _\ga ) \bigcap H^2 (\B ^d _\N ), $$
is then Cauchy in $\scr{H} ^+ (H _\ga )$,
\ba \| h_n - h_m \| _{H_\ga} & \leq & \| \mr{e} _\mu (f_n - f_m ) \| _{H_\ga} + \| \mr{e} _\la (g_n -g_m ) \| _{H_\ga} \nn \\
& \leq & \| f_n - f_m \| _{H_\mu} + \| g_n - g_m \| _{H _\la} \rightarrow 0. \nn \ea
For any $Z \in \B ^d _\N$, 
$$ h_n (Z) = f_n (Z) + g_n (Z) \rightarrow f(Z) + g(Z) = h(Z), $$ and it follows that $h$ is the limit of the Cauchy sequence $(h_n)$. This proves that
$$ H^2 (\B ^d _\N ) \bigcap \scr{H} ^+ (H _\ga ), $$ is dense in $\scr{H} ^+ (H_\ga )$, and $\ga = \la +\mu$ is then an absolutely continuous NC measure.
\end{proof}

\begin{lemma} \label{singhered}
The set of singular NC measures is hereditary: If $\mu \in \mr{Sing} (\A ^\dag ) _+$, $\la$ is any positive NC measure and $\mu \geq \la$, then $\la$ is also singular.
\end{lemma}
\begin{proof}
    If $\la$ is not singular then $\mu \geq \la \geq \la _{ac} \neq 0$. It follows that 
$$ \{ 0 \} \subsetneqq \scr{H} ^+ (H _{\la _{ac} } ) \bigcap H^2 (\B ^d _\N ) \subset \scr{H} ^+ (H _\mu ), $$ so that the space of free Cauchy Transforms of $\mu$ has non-trivial intersection with the Free Hardy Space.  This contradicts the assumption that $\mu$ is singular.
\end{proof}

\section{AC measures and closable $L-$Toeplitz forms}
\label{ACformsect}

Any positive NC measure $\mu \in (\A ^\dag ) _+$ can be identified with a densely-defined, positive semi-definite quadratic form, $q _\mu$ on the Fock space. In \cite{JM-NCFatou} we applied B. Simon's Lebesgue decomposition theory for quadratic forms to $q_\mu$ \cite[Section 2]{Simon1} to construct a non-commutative Lebesgue decomposition of any NC measure into absolutely continuous and singular parts. In this section we prove that this `Lebesgue form decomposition' of any $\mu \in (\A ^\dag ) _+$ and the Lebesgue decomposition developed in the previous section using (NC) reproducing kernel techniques are the same. A good reference for the theory of potentially unbounded quadratic forms on Hilbert space is \cite{Kato}, see also \cite[Section VIII.6]{RnS1}. We refer to \cite[Section 4]{JM-NCFatou} for more detail on the quadratic forms arising from NC measures.

\begin{defn} \label{LToeform}
A densely-defined positive semi-definite quadratic (sesquilinear) form, $q$, with dense domain $\dom{q} := \A \subseteq F^2 _d$ is called 
an \emph{$L-$Toeplitz form} if there is a (positive) NC measure, $\mu \in (\A ^\dag ) _+$, so that 
$$ q ( a_1  , a_2  ) = \mu \left( (a_1 (L) ) ^* a_2 (L) \right) =: q_\mu (a_1 , a_2); \quad a_1, a_2 \in \A.$$  

Given any positive semi-definite quadratic form, $q$, with dense form domain $\dom{q} = \A \subset F^2 _d$, we define the (generally non-positive) linear functional, $\hat{q} : \A + \A ^* \rightarrow \C$, by 
$$ \hat{q} (a_1 +a_2 ^*) := q (1 , a_1 ) + q (a_2 , 1 ). $$ 
\end{defn}

Recall, that we defined closed positive semi-definite quadratic forms in Subsection \ref{classmeas}, and that a positive semi-definite quadratic form, $q$, with dense domain in $\H$ is closed if and only if
$$ q (h ,g ) = q_A (h, g) := \ip{ \sqrt{A} h }{\sqrt{A} g} _\H; \quad \quad g,h \in \dom{q} =\dom{\sqrt{A}}, $$
for some closed, positive semi-definite operator $A$.  A positive quadratic form, $q$, is \emph{closable} if it has a closed extension. If $q$ is closable, then it has a minimal closed extension, $\ov{q}$, with $\dom{\ov{q}} \subseteq \H$ equal to the set of all $h \in \H$ so that there is a sequence $h_n \in \dom{q}$, such that $h_n \rightarrow h$ and $(h_n)$ is Cauchy in the norm of $\H (q+1)$. A dense set $\mc{D} \subseteq \dom{q}$ is called a form core for a closed form $q$ if $\mc{D}$ is a dense linear subspace in $\H (q +1 )$. If $q$ is closable with closure (minimal closed extension) $\ov{q}$, then $\dom{q}$ is a form core for $\ov{q}$ \cite[Chapter VI, Theorem 1.21]{Kato}. If $q = q_A$ is a closed, positive semi-definite quadratic form, then $\mc{D}$ is a form core for $q$ if and only if $\mc{D}$ is a core for $\sqrt{A}$. In particular, $\dom{A}$ is a form core for $q$.

\begin{defn} \label{LToepdef}
A closed, positive semi-definite operator $T$ with domain $\dom{T} \subseteq F^2 _d$ will be called $L-$Toeplitz if:
\bn 
\item $\dom{\sqrt{T}}$ is $L-$invariant and $\fp \subseteq \dom{\sqrt{T}}$ is a core for $\sqrt{T}$,
\item The associated quadratic form 
$$ q _{T} ( a_1, a_2 ) := \ip{\sqrt{T} a_1 (L) 1 }{\sqrt{T} a_2 (L) 1} _{F^2 _d}; \quad \quad a_1, a_2 \in \A $$ is $L-$Toeplitz.
\en
\end{defn}
If $T$ is a bounded $L-$Toeplitz operator, then
$$ L_k ^* T L_j = \delta _{k,j} T, $$ so that $T$ is $L-$Toeplitz in the sense of Popescu, see \cite[Section 1.1]{Pop-entropy}. In \cite[Section 2]{Simon1}, B. Simon proved that any densely-defined and positive semi-definite quadratic form, $q$, acting in a Hilbert space $\H$, has a unique Lebesgue Decomposition: $$ q = q_{ac} + q_s, $$ 
where $q_{ac}$ is the maximal closable form dominated by $q$, and $q_s = q - q_{ac}$. It follows that any $\mu \in (\A ^\dag ) _+$  has the \emph{Lebesgue Form Decomposition}: 
\be \mu =  \hat{q} _{ac} + \hat{q} _s, \label{LDform} \ee where $\hat{q} _{ac}, \hat{q} _s $ are (a priori not necessarily positive) linear functionals on the free disk system, see Definition \ref{LToeform}. By \cite[Equation (5.2)]{JM-NCFatou}, the NC measure $\hat{q} _{ac} \in (\A ^\dag ) _+$ is given by the formula:
\be \hat{q} _{ac} (L^\alpha ) = \ip{(I + N_{\mu +m})}{(I -Q) \Pi _{\mu +m} ^\alpha (I + N_{\mu +m})}_{\mu +m}, \label{formACpart} \ee where 
$Q$ is the orthgonal projection onto the kernel of the contractive embedding $E : F^2 _d (\mu +m ) \hookrightarrow F^2 _d$. In \cite[Theorem 5.9]{JM-NCFatou}, we proved that $ \hat{q} _{ac}$ and $ \hat{q} _s$ are positive NC measures, so that this yields a `quadratic form' Lebesgue decomposition of $\mu$ and an altenative definition of `absolutely continuous' and `singular' positive NC measures. (The next theorem shows that these potentially different decompositions and definitions are the same.) 

\begin{thm} \label{ACismaxclosed}
An NC measure $\mu \in (\A ^\dag ) _+$ is absolutely continuous if and only if $q_\mu \geq 0$ is a closable quadratic form.  If $\mu$ is absolutely continuous and $q =q_T$ is the closure of $q_{\mu}$, then the positive semi-definite operator $T$ is $L-$Toeplitz.
\end{thm}

\begin{proof}
By \cite[Corollary 5.6]{JM-NCFatou}, an NC measure $\mu \in (\A ^\dag ) _+$ generates a closable quadratic form, $q_\mu$ if and only if the intersection of the space of NC Cauchy Transforms of $\mu + m$ with the NC Hardy space is dense in $\scr{H} ^+ (H _{\mu +m} )$, that is, if and only if $\mu +m$ is an absolutely continuous NC measure in the sense of Definition \ref{ncldDef}. 

We claim that $\mu +m$ is absolutely continuous if and only if $\mu$ is absolutely continuous so that these two definitions of absolute continuity are equivalent. First, by Theorem \ref{ACcone}, $AC (\A ^\dag ) _+$ is a positive cone so that if $\mu$ is AC, so is $\mu +m$. Conversely, if $\mu +m$ is AC, this is equivalent to $q _{\mu }$ being a closable quadratic form, so that 
$\ov{q_\mu} =: q_T$ is the quadratic form of a unique, positive semi-definite, $L-$Toeplitz $T \geq 0$, and $\fp$ is a core for $\sqrt{T}$ by \cite[Theorem 5.8]{JM-NCFatou}. Suppose that $x \in \dom{T} \subseteq \dom{\sqrt{T}}$. Then since $\fp$ is a core for $\sqrt{T}$, we can find a sequence of free polynomials, $p_n$, so that 
$$ p_n \rightarrow x, \quad \mbox{and} \quad \sqrt{T} p_n \rightarrow \sqrt{T} x. $$ 
In particular, the sequence $p_n (L) + N _\mu$ is Cauchy in $F^2 _d (\mu )$, and converges to a vector $\hat{x} \in F^2 _d (\mu )$:
$$ \| p_n - p_m + N _\mu \| _{\mu} =  \| \sqrt{T} (p_n - p_m ) \| _{F^2 _d} \rightarrow 0. $$
It follows that we can identify $\dom{T}$ with a linear subspace (generally non-closed), $\mc{D} _\mu (T) \subset F^2 _d (\mu )$. We claim that any vector $y \in \mc{D} _\mu (T)$ is such that
$\mc{C} _\mu y \in H^2 (\B ^d _\N )$. Indeed, as above, given $y \in \mc{D} _\mu (T)$, there is a vector $\check{y} \in \dom{T}$ and a sequence of free polynomials $p_n$ so that $p_n \rightarrow \check{y}$, $\sqrt{T} p_n \rightarrow \sqrt{T} \check{y}$, and  $p_n (L) + N_\mu \rightarrow y$ in $F^2 _d (\mu)$. The free Cauchy Transform of $y$ is:
\ba ( \mc{C} _\mu y ) (Z) & = & \sum _{\alpha \in \F ^d } Z^\alpha \ip{\Pi _\mu ^\alpha (I + N_\mu )}{y}_\mu \nn \\
& = & \lim _{n \rightarrow \infty} \sum Z^\alpha \ip{\Pi _\mu ^\alpha (I + N_\mu )}{p_n (L) + N_\mu} _\mu \nn \\
& = & \lim _n \sum Z^\alpha \ip{\sqrt{T} L^\alpha 1}{\sqrt{T} p_n (L) 1} _{F^2 _d} \nn \\
& = &  \sum Z^\alpha \ip{\sqrt{T} L^\alpha 1}{\sqrt{T} \check{y}}_{F^2 _d} \nn \\
& = & \sum Z^\alpha \ip{L^\alpha 1}{T \check{y}}_{F^2 _d}. \nn  \\
& = & (T\check{y}) (Z). \nn \ea 
Since $T\check{y} \in H^2 (\B ^d _\N ) = F^2 _d$ this proves our claim. Moreover, by general facts about closed operators, $\dom{T}$ is a core for $\sqrt{T}$, and it follows that $\mc{D} _\mu (T)$ is norm-dense in $F^2 _d (\mu )$. This proves that 
$$ \scr{H} ^+ (H _\mu ) \bigcap H^2 (\B ^d _\N ), $$ is dense in $\scr{H} ^+ (H _\mu )$, so that $\mu$ is, by definition, an absolutely continuous NC measure.
\end{proof}

\begin{thm}  \label{sameNCLD}
Given $\mu \in (\A ^\dag ) _+$, the Lebesgue Form Decomposition and Lebesgue Decomposition of $\mu$ coincide. That is, the quadratic form of $\mu _{ac}$ is the maximal closable quadratic form bounded above by $q_\mu$.
\end{thm}

\begin{lemma}
Given $\mu \in (\A ^\dag ) _+$ with Lebesgue Decomposition $\mu = \mu _{ac} + \mu _s$, if $\la = \mu + m$, then $\la$ has Lebesgue decomposition:
$$ \la = \underbrace{\mu_{ac} +m}_{=\la _{ac}} + \underbrace{\mu _s}_{=\la _s}.$$
\end{lemma}
\begin{proof}
By Theorem \ref{NCintersect}, 
$$ \scr{H} ^+ (H _{\la}) = \scr{H} ^+ (H_{\mu _{ac}}) + \underbrace{\scr{H} ^+ (H _m )}_{=H^2 (\B ^d _\N )} + \scr{H} ^+ (H _{\mu _s} ), $$ and each of the spaces of this decomposition is contractively contained in $\scr{H} ^+ (H _\la )$, with $\la = \mu + m$. Since $AC (\A ^\dag ) _+$ is a positive cone, by Theorem \ref{ACcone}, $\mu _{ac} + m$ is AC. One can show, as in the proof of Theorem \ref{ACcone},
$$ H^2 (\B ^d _\N) \cap  \scr{H} ^+ (H _{\mu _{ac} +m } ), $$ is dense in the subspace
$$ \left( \scr{H} ^+ (H_{\mu _{ac} +m } ) \right) ^{-\| \cdot \| _{H_\la}}, $$ and it follows that $\mu _{ac} + m \leq (\mu + m ) _{ac} = \la _{ac}$. Also, since $\mu _{s}$ is the singular part of $\mu$, we know that both
$$ \scr{H} ^+ (\mu _s ) \bigcap \scr{H} ^+ (\mu _{ac} ) = \{ 0 \}, $$ by Theorem \ref{ncld2thm}, and 
$$ \scr{H} ^+ (\mu _s ) \bigcap H^2 (\B ^d _\N ) = \{ 0 \}, $$ by definition. We claim also that
$$ \scr{H} ^+ (\mu _s ) \bigcap \scr{H} ^+ ( H_{\mu _{ac} +m } ) = \{ 0 \}. $$ 
Indeed, we have as above that
$$ \scr{H} ^+ (H _{\mu _{ac} +m } ) = \scr{H} ^+ ( H_{\mu _{ac}} ) + H^2 (\B ^d _\N ), $$ 
as vector spaces, so that if $f \in \scr{H} ^+ (\mu _s ) \bigcap \scr{H} ^+ (H _{\mu _{ac} +m })$, then 
$$ f = g + h; \quad \quad g \in \scr{H} ^+ (H _{\mu _{ac}} ), \ h \in H^2 (\B ^d _\N ). $$ 
However, this would imply that 
$$ f-g = h \in H^2 (\B ^d _\N ) \bigcap \scr{H} ^+ (H _\mu ) \subseteq \scr{H} ^+ (H _{\mu _{ac}} ), $$ by the definition of the absolutely continuous part of $\mu$, so that $g, f-g$, and hence $f$ belong to $\scr{H} ^+ ( H_{\mu _{ac} } )$. Since the Herglotz space of $\mu _{ac}$ is by construction orthogonal to $\scr{H} ^+ (H_{\mu _s} )$, $f = 0$, and this proves that the intersection of $\scr{H} ^+ (H _{\mu _s} )$ with $\scr{H} ^+ (H _{\mu _{ac} +m } )$ is empty. 
By Theorem \ref{NCintersect} we then have the direct sum decompositions:
$$ \begin{array}{llll} \scr{H} ^+ (H_\la ) & = \scr{H} ^+ (H _{\mu _{ac} +m } ) & \oplus & \scr{H} ^+ (H _{\mu _s} ) \\ 
 & = \scr{H} ^+ ( H_{\la _{ac}} ) & \oplus & \scr{H} ^+ (H _{\la _s} ). 
\end{array} $$ 
The first decomposition, implies, in particular, that $\scr{H} ^+ (H_{\mu _{ac} +m } )$ is contained isometrically in $\scr{H} ^+ (H _\la )$, and since $\mu _{ac} + m \leq \la _{ac}$, it is contained isometrically inside $\scr{H}^+ (H _{\la _{ac} } )$. However, by definition,
$$ \scr{H} ^+ (H _{\la _{ac} } ) \bigcap H^2 (\B ^d _\N )$$ is dense in $\scr{H} ^+ (H _{\la _{ac}} )$, and 
\ba  \scr{H} ^+ (H _{\la _{ac} } ) \bigcap H^2 (\B ^d _\N ) & \subseteq & H^2 (\B ^d _\N ) \\ 
& \subseteq & H^2 (\B ^d _\N ) + \scr{H} ^+ (H _{\mu _{ac}} ) \nn \\
& = & \scr{H} ^+ (H _{\mu _{ac} +m } ), \nn \ea so that 
$$ \scr{H} ^+ (H _{\la _{ac} } ) \bigcap \scr{H} ^+ ( H_{\mu _{ac} +m } ) $$ is dense in 
$\scr{H} ^+ (H _{\la _{ac} } )$. Since these are both closed subspaces, it must be that $\la _{ac} = \mu _{ac} +m$ and $\mu _s = \la _s$.
\end{proof}

\begin{proof}{ (of Theorem \ref{sameNCLD})}
It remains to prove that if $\mu = \mu _{ac} + \mu _s$ is the Lebesgue decomposition of $\mu$ of Theorem \ref{ncld2thm}, that $\mu _{ac}$ generates the largest closable quadratic form bounded above by $\mu$, so that $\mu _{ac} = \hat{q} _{ac}$ and the Lebesgue decomposition and Lebesgue form decompositions of $\mu$ coincide. Let $\mr{e} : H^2 (\B ^d _\N ) = \scr{H} ^+ (H _m) \hookrightarrow \scr{H} ^+ (H _{\mu +m } )$ be the contractive embedding (since $m \leq \mu + m$). By Lemma \ref{embeddings},
$$ E = \mc{C} _m ^* \mr{e} ^* \mc{C} _{\mu +m}, $$ and it follows that the kernel of $E$ is the kernel of $\mr{e}^* \mc{C} _{\mu + m}$. 

By Theorem \ref{NCsum}, $\scr{H} ^+ (H _{\mu _{ac} +m} ) = \scr{H} ^+ (H _{\mu _{ac}} ) + H^2 (\B ^d _\N )$ and the NC Hardy space is contractively contained in $\scr{H} ^+ (H _{\mu _{ac} +m})$. Furthermore, by Theorem \ref{ACcone}, $\mu _{ac} +m$ is an absolutely continuous NC measure so that $H^2 (\B ^d _\N) \subseteq H^2 (\B ^d _\N ) \bigcap \scr{H} ^+ (H_{\mu _{ac} +m})$ is norm-dense in the space of $(\mu _{ac} +m)-$Cauchy transforms. Since the previous lemma implies that $(\mu +m ) _{ac} = \mu _{ac} +m$, it follows that the range of $\mr{e}$ is contained in and norm-dense in $\scr{H} ^+ (H _{(\mu +m) _{ac} } )$ so that 
$$ \ov{\ran{\mr{e}}} = \scr{H} ^+ (H _{(\mu +m ) _{ac}} ). $$ Consequently, and again by the previous lemma, 
$$ \ov{\ran{\mr{e}}} ^\perp = \ker{\mr{e} ^*} = \scr{H} ^+ (H _{(\mu +m ) _s } ) = \scr{H} ^+ (H _{\mu _s} ), $$ and 
$$ \ker{E} = F^2 _d (\mu _s ). $$ By Formula (\ref{formACpart}), it follows that $\hat{q} _{ac} = \mu _{ac}$. 
\end{proof}

\section{Lebesgue Decomposition for Row Isometries}
\label{NCLDsect}

The concept of absolute continuity, singularity, and Lebesgue Decomposition for bounded linear functionals on $\A$ was first defined and studied in the context of free semigroup algebra theory \cite{DLP-ncld,DKP-structure,MK-rowiso}.  Recall, a free semigroup algebra
is any $WOT-$closed unital operator algebra generated by a row isometry. If $\Pi$ is a row isometry on a Hilbert space, $\H$, we denote the free semigroup algebra of $\Pi$ by
$$ \F ^d (\Pi) := \mr{Alg} (I, \Pi ) ^{-WOT}. $$ As proven in \cite{DP-inv}, the weak$-*$, and WOT-closures of $\A$ coincide so that the left free analytic Toeplitz algebra, $L^\infty _d = H^\infty (\B ^d _\N )$, is a free semigroup algebra. 
\begin{defn}{ (see \cite[Definition 2.1]{DLP-ncld} and \cite[Theorem 2.10]{DP-inv})} \label{posAC}
A bounded linear functional $\varphi  \in \A ^\dag$ is weak$-*$ continuous if it has a weak$-*$ continuous extension to $L^\infty _d$.
\end{defn}
\begin{thm}{ (\cite[Theorem 2.10]{DP-inv})}
A bounded linear functional $\phi \in \A ^\dag$ is weak$-*$ continuous if and only if there are vectors, $x,y \in F^2 _d$ so that 
$$ \phi  (a) = m_{x,y} (a) := \ip{x}{a (L ) y} _{F^2 _d}. $$ 
\end{thm}
A natural extension of the above definition to (positive) NC measures on the Free Disk System is then:
\begin{defn} \label{DLP-AC}
A bounded positive linear functional (or NC measure) $\phi \in (\A ^\dag) _+$, is weak$-*$ continuous if it has a weak$-*$ continuous extension to the (left) \emph{Toeplitz System} $\left( L^\infty _d + (L^\infty _d )^* \right) ^{-\mr{WOT}} = \left( \A + \A ^* \right) ^{-\mr{WOT}}$. Let $WC (\A  ^\dag ) _+$ denote the positive cone of all weak$-*$ continuous NC measures.
\end{defn}
Clearly $WC (\A ^\dag ) _+$ is a positive cone since positive linear combinations of positive weak$-*$ continuous linear functionals are again weak$-*$ continuous and positive. We will prove that any (positive) NC measure is weak$-*$ continuous in the above Davidson-Li-Pitts sense if and only if it is absolutely continuous in the sense of Definition \ref{ncldDef}, see Theorem \ref{sameLD}.

\begin{defn} \label{starext}
A representation $\pi : \A  \rightarrow \mc{L} (\H )$ on a separable Hilbert space, $\H$, is called $*-$extendible if and only if it is the restriction of a unital $*-$representation of the Cuntz-Toeplitz $C^*$-algebra, $\mc{E} _d = C^* (I , L)$ to $\A $. 
\end{defn}
A unital homomorphism $\pi : \A \rightarrow \mc{L} (\H)$ is $*-$extendible if and only if $\Pi _k := \pi (L_k)$ is a row isometry.  The following concept of a weak$-*$ continuous vector will be important for our investigations:
\begin{defn}{ (\cite[Definition 2.4]{DLP-ncld})} \label{WCdef}
A vector $h \in \H$ is called a weak$-*$ continuous (WC) vector for a $*-$representation, $\pi$, of $\mc{E} _d$, if
$$ \phi _h (L^\alpha ) := \ip{h}{\pi (L ) ^\alpha h} _\H, $$ is a weak$-*$ continuous functional on $\A $. The set of all weak$-*$ continuous vectors for $\pi$ is denoted by $WC (\pi)$, or $WC (\mu)$ if $\pi = \pi _\mu$ is the GNS representation of an NC measure.
\end{defn}
\begin{defn}{ (\cite[Definition 2.6]{DLP-ncld})} \label{intertwiner}
A bounded linear map $X : F^2 _d \rightarrow \H$ is called an \emph{intertwiner} for a $*-$extendible representation $\pi$ if $X L^\alpha = \Pi ^\alpha X$. The set of all intertwiners is denoted $\chi (\pi)$ (or $\chi (\mu)$ if $\pi = \pi _\mu$ for an NC measure $\mu$).
\end{defn}
Weak$-*$ continuous vectors are characterized by the following theorem \cite[Theorem 2.7]{DLP-ncld}:
\begin{thm} \label{DLP-intertwiner}
Let $\pi$ be a $*-$extendible representation of $\A $ on $\H$. Then $WC (\pi)$ is a $\Pi := \pi (L)-$invariant, closed subspace and $WC (\pi) = \chi (\pi) F^2 _d$. Given any $x,y \in WC (\pi )$, $$ \mu _{x,y} (p(L)) := \ip{x}{\pi (p(L)) y } _\H; \quad \quad p \in \fp, $$ defines a weak$-*$ continuous functional on $\A $. 
\end{thm}
In \cite{MK-rowiso}, M. Kennedy extended and applied these notions to develop a Lebesgue decomposition of row isometries. Namely, let $\Pi$ denote an arbitrary row-isometry on a Hilbert space $\H$. By the Kennedy-Wold-Lebesgue decomposition $\Pi$ and $\H$ decompose as direct sums:
$$ \Pi =: \Pi _L \oplus \Pi _{C-L} \oplus \Pi _{vN} \oplus \Pi _{dil}, $$ on 
$$ \H = \H _L \oplus \H _{C-L} \oplus \H _{vN} \oplus \H _{dil}. $$ where $\Pi _L$ is \emph{pure type$-L$}, $\Pi _{C-L}$ is called \emph{Cuntz type$-L$}, $\Pi _{vN}$ is purely singular or of \emph{von Neumann type}, and $\Pi _{dil}$ is of \emph{dilation type}. These classes of row isometries are defined as follows:

\begin{defn} \label{rowisotype}
	A row-isometry, $\Pi$, on $\H$ is:
\bn
\item \emph{type$-L$} if it is unitarily equivalent to a vector-valued left free shift $L \otimes I_\K$ for some Hilbert space $\K$. 
\item \emph{Cuntz type$-L$}  if it is an onto row isometry (also called a Cuntz unitary) and the free semigroup algebra generated by $\Pi$, $\F ^d (\Pi ) = \mr{Alg} (I, \Pi ) ^{-WOT}$, is isomorphic to $L^\infty _d$, \emph{i.e.} if the map $\Pi _k \mapsto L_k$ extends to a completely isometric isomorphism and weak$-*$ continuous homeomorphism of $\F ^d (\Pi)$ onto $L^\infty _d$.
\item \emph{weak$-*$ continuous} (WC),  if it is a direct sum of type-L and Cuntz type-L row isometries. 
\item \emph{von Neumann type} if it has no weak$-*$ continuous restriction to an invariant subspace.
\item \emph{dilation type} if $\Pi$ has no direct summand which is one of the previous types. 
\item \emph{weak$-*$ singular} (WS) if $\Pi $ is a direct sum of von Neumann and dilation-type row isometries.
\en
\end{defn}

\begin{remark}
Von Neumann and dilation-type row isometries are necessarily Cuntz unitary. Any dilation-type row isometry can be decomposed in the form:
$$ \Pi \simeq \bpm T & 0 \\ * & L \otimes I_\H \epm, $$ (so that the restricion of $\Pi$ to an invariant subspace is unitarily equivalent to several copies of $L$). As shown in \cite{MK-rowiso}, $\Pi$ is of von Neumann type if and only if the $WOT-$closed algebra generated by $\Pi$ (\emph{i.e.} the free semigroup algebra of $\Pi$) is self-adjoint, \emph{i.e.} a von Neumann algebra. Von Neumann type row isometries are at this point rather mysterious and poorly understood. There is essentially only one known example of a von Neumann type row isometry due to C. Read \cite{Read,KRD-Read} which constructs an example of a two-component row isometry $\Pi = (\Pi _1, \Pi _2)$ on a separable Hilbert space, $\H$, so that the $WOT-$closed algebra generated by $\Pi$ is all of $\mc{L} (\H )$.  In particular, it is unknown whether one can generate other types of von Neumann algebras in this way.
\end{remark}
\begin{remark}
In the free semigroup algebra literature, several variations of the concept of a weak$-*$ continuous row isometry (as we have defined it above) or $*-$representation of $\mc{E} _d$ were introduced in \cite{DLP-ncld} to describe when the weak$-*$ closure of a free semigroup algebra of a row isometry or Cuntz-Toeplitz $*-$representation is similar in structure to $L^\infty _d$, see \cite[Theorem 3.4]{DLP-ncld}. There is also no clear consensus on terminology see \emph{e.g.} \cite[Theorem 3.4]{DLP-ncld} and \cite[Definition 3.2, Definition 3.6]{MK-rowiso}. Eventually, the work of several authors showed that all these variations of type$-L$ row isometries were the same \cite[Definition 3.1, Theorem 3.4]{DLP-ncld}, \cite[Definition 3.2, Definition 3.6, Theorem 4.16]{MK-rowiso}, \cite{DY-freesemi}:
\begin{thm} \label{ACrowiso}
Let $\Pi$ be a row isometry on a Hilbert space, $\H$. The following are equivalent:
\bn
    \item $\Pi$ is weak$-*$ continuous.
    \item The representation $L_k \mapsto \Pi _k$ induced by $\Pi$ is the restriction to $\A$ of a weak$-*$ continuous representation of $L^\infty _d$.
    \item Every vector in $\H$ is a weak$-*$ continuous vector for $\Pi$, $\H = WC (\Pi )$.
\en
\end{thm}
\begin{proof}
The equivalence of the first two items is \cite[Theorem 4.16]{MK-rowiso} (see also Definitions $3.2$ and $3.6$). If $\Pi$ is a weak$-*$ continuous row isometry (as we have defined it) then the fact that $WC (\Pi ) = \H$ follows from \cite[Theorem 3.4]{DLP-ncld}, or equivalently from \cite[Theorem 4.17]{MK-rowiso} which proves the stronger statement that $\H$ is spanned by wandering vectors for $\Pi$.

Conversely, the main result of \cite{DY-freesemi} is that if $\H = WC (\Pi )$, then the infinite ampliation, $\Pi ^{(\infty)} \simeq \Pi \otimes I _{\ell ^2 (\N _0)}$ is a weak$-*$ continuous row isometry. In this case, as observed in \cite{DY-freesemi}, the weak$-*$ closure of the free semigroup algebra of $\Pi$ is completely isometrically isomorphic and weak$-*$ homeomorphic to the free semigroup algebra of $\Pi ^{(\infty)}$ (recall a general free semigroup algebra is a priori only $WOT-$closed, not necessarily weak$-*$ closed, by definition), and hence to $L^\infty _d$, since $\Pi ^{(\infty )}$ is weak$-*$ continuous. However, this implies that the representation $\pi : \A \rightarrow \mc{L} (\H )$ induced by $\Pi$, $\pi (L_k ) := \Pi _k$, is the restriction of a weak$-*$ continuous representation of $L^\infty _d$, and hence by \cite[Definition 3.2, Definition 3.6, Theorem 4.16]{MK-rowiso}, the free semigroup algebra of $\Pi$ is isomorphic to $L^\infty _d$. As described in \cite{DY-freesemi} algebraic isomorphism necessarily implies the much stronger property that they are completely isometrically isomorphic and weak$-*$ homeomorphic. By Definition \ref{rowisotype} above, $\Pi$ is then a weak$-*$ continuous row isometry. 
\end{proof}
\end{remark}

We now apply the Kennedy-Wold-Lebesgue decomposition of row isometries to (positive) NC measures:
\begin{defn} \label{NCLDdef2}
Given $\mu \in (\A ^\dag) _+$, we say $\mu$ is one of the six types of Definition \ref{rowisotype} if its GNS row isometry $\Pi _\mu$ is of that corresponding type. The Kennedy-Wold-Lebesgue decomposition of $\mu$ is:
$$ \mu = \mu _L + \mu _{C-L} + \mu _{vN} + \mu _{dil}, $$ where each $\mu _{\mr{type}}  \in (\A  ^\dag ) _+$ is positive and bounded above by $\mu$ and
$$ F^2 _d (\mu ) = F^2 _d (\mu ) _L \oplus F^2 _d (\mu )_{C-L} \oplus F^2 _d (\mu ) _{vN} \oplus F^2 _d (\mu ) _{dil}, $$ is the Kennedy-Wold-Lebesgue direct sum decomposition. If $P _{L}, P_{C-L}, P_{vN}, P _{dil}$ are the corresponding reducing projections, 
$$ \mu _{\mr{type}} ( \cdot  ) := \ip{I+N_\mu }{\pi _\mu (\cdot ) P _{\mr{type}} (I + N_\mu) } _\mu, $$ where $\ip{\cdot}{\cdot} _\mu$ is the GNS inner product of $\mu$ and $\mr{type} \in \{ L, C-L, wc, vN, dil, ws \}$.

The weak$-*$ Lebesgue Decomposition  of $\mu$ is then: 
$$ \mu =:  \underbrace{\mu _L + \mu _{C-L} }_{=: \mu _{wc}} +\underbrace{\mu _{vN} + \mu _{dil} }_{=: \mu _{ws} } = \mu _{wc} + \mu _{ws}, $$ $\mu _{wc}, \mu _{ws} \in (\A ^\dag ) _+$ are called the weak$-*$ continuous and weak$-*$ singular parts of $\mu$, respectively, and are both bounded above by $\mu$. We will let $WC(\A ^\dag ) _+$, $WS(\A ^\dag ) _+$ denote the sets of weak$-*$ continuous and weak$-*$ singular NC measures, respectively. 

Similarly we write $F ^2 _d(\mu ) _{wc} := F^2 _d (\mu) _L \oplus F^2 _d (\mu) _{C-L}$ and $F^2 _d (\mu) _{ws} = F^2 _d (\mu) _{vN} \oplus F^2 _d (\mu ) _{dil}$ so that $F^2 _d (\mu ) _{wc}$ and $F^2 _d (\mu ) _{ws}$ are reducing subspaces for $\Pi _\mu $ with orthogonal projections $P_{wc} = P_L \oplus P _{C-L}$, $P_{ws} = P_{vN} \oplus P_{dil}$ and then 
$$ F^2 _d (\mu) = F^2 _d (\mu ) _{wc} \oplus F^2 _d (\mu  ) _{ws}. $$ 
\end{defn}
The spaces $F^2 _d (\mu ) _{\mr{type}}$ and $ F^2 _d (\mu _{\mr{type}})$ are naturally isomorphic. We will ultimately show that $\mu _{wc} = \mu _{ac}$ and $\mu _{ws} = \mu _s$ so that Lebesgue Decomposition and weak$-*$ Lebesgue Decomposition of any positive NC measure coincide.

\begin{cor}
The weak$-*$ continuous subspace, $F^2 _d (\mu _{wc} ) \subseteq F^2 _d (\mu )$, is the largest $\Pi _\mu -$reducing subspace of weak$-*$ continuous vectors for $\mu$. The $\Pi _\mu -$invariant subspace of WC vectors for $\mu$ is $WC (\mu ) = F^2 _d (\mu _{wc}) \oplus \left( F^2 _d (\mu _{dil} ) \cap WC (\mu ) \right). $
\end{cor}
This is an immediate consequence of Theorem \ref{ACrowiso} and the definitions.

\begin{remark}
It is natural that the weak$-*$ continuous part of an NC measure $\mu $ should include $\mu _L + \mu _{C-L}$, and that the weak$-*$ singular part of $\mu$ should include $\mu _{vN}$. It may not seem immediately obvious that the dilation part of $\mu$ should be included in the singular part of $\mu$ since any dilation-type row isometry has a weak$-*$ continuous restriction to an invariant subspace by definition (\emph{i.e.} it has weak$-*$ continuous vectors). However, our results will show that this definition is consistent and justified.

If $\mu \in (\A ^\dag ) _+$ is an NC measure, our weak$-*$ Lebesgue Decomposition of $\mu$ differs from the Lebesgue Decomposition for $\mu | _{\A }$, as defined in \cite[Proposition 5.9]{DLP-ncld}. Indeed, the Davidson-Li-Pitts Lebesgue decomposition of $\mu$ as a functional on $\A$ is: $\mu = \check{\mu} _{wc} + \check{\mu} _s$, where 
$$ \check{\mu } _{wc} (L^\alpha ) = \ip{I}{\Pi _\mu ^\alpha Q_{wc} I }_\mu, \quad \mbox{and} \quad \check{\mu } _{wc} (L^\alpha ) = \ip{I}{\Pi _\mu ^\alpha Q_{ws} I }_\mu, $$ $Q _{wc}$ is the projection onto the invariant subspace of all weak$-*$ continuous vectors for $\pi _\mu$, and $Q_{ws} = I - Q_{wc}$. This differs from our weak$-*$ Lebesgue Decomposition, in general, since our $P_{wc} = P_L + P _{C-L} \leq Q_{wc}$. The decompositions are the same if and only if $\Pi _\mu$ has no direct summand of dilation-type.

As Theorem \ref{sameLD} will show, the $\mu _{wc}$ from our decomposition is the maximal weak$-*$ continuous functional which is both positive and bounded above by the original NC measure $\mu$. One can check that if $\mu$ is a positive NC measure, that (since $Q_{wc}$ is $\Pi _\mu -$invariant) the functional $\check{\mu} _{wc}$ extends to a positive NC measure on $\A + \A ^*$:
$$ \check{\mu} _{wc} (a^* a) = \ip{ I + N _\mu }{Q_{wc} \pi _\mu (a^* a) Q_{wc} (I +N_\mu) } _\mu. $$ However, the operator $$\pi _\mu (a) ^* \pi _\mu (a) - Q_{wc} \pi _\mu (a) ^* \pi _\mu (a) Q_{wc}, $$ need not be positive semi-definite, so that $\check{\mu} _{wc} $ need not be bounded above by the original NC measure $\mu$.  Indeed, since our $\mu _{wc} = \mu _{ac}$ is the maximal absolutely continuous NC measure bounded above by $\mu$ (see Theorem \ref{sameLD}), it must be that $\check{\mu} _{wc}$ is not bounded above by $\mu$ unless $\check{\mu} _{wc} = \mu _{wc} \ ( = \mu _{ac} ) $ and $(P_{ac} =) \ P_{wc} = Q_{wc}$ is reducing for $\Pi _\mu$.
\end{remark}

\begin{cor} \label{ACfun}
An NC measure $\mu \in (\A ^\dag ) _+$ is weak$-*$ continuous  if and only if it is given by a positive vector functional on the Fock Space, \emph{i.e.} $\mu = m_{x,y} = m_{y,x} \geq 0$ for $x,y \in F^2 _d$ where
$$ m_{x,y} (L^\alpha ) := \ip{x}{L^\alpha y}_{F^2 _d}.$$ Equivalently $\mu$ is weak$-*$ continuous  if and only if $\Pi _\mu$ is a weak$-*$ continuous row isometry.
\end{cor}

Any strictly positive $L-$Toeplitz operator which is bounded above and below has an analytic outer factorization:
\begin{thm}{ (Popescu \cite[Theorem 1.5]{Pop-entropy})} \label{Popfactor}
Any positive $L-$Toeplitz $T \in \mc{L} (F^2 _d )$ which is bounded below, $T \geq \eps I$ can be factored as:
$T = F (R) ^* F (R)$ for some outer $F \in R^\infty _d $.  
\end{thm}
If $T\geq 0$ is an arbitrary positive semi-definite $L-$Toeplitz operator, it is still possible to obtain an asymmetric factorization $T = F(R) ^* G(R) = G(R) ^* F(R)$ with $F,G \in R^\infty _d$:
\begin{lemma}{ (\cite[Lemma 3.2, Lemma 3.3]{MK-wand})} \label{MK-factor}
If $d \geq 2$, $R^\infty _d + (R^\infty _d) ^*$ is precisely equal to the set of bounded $L-$Toeplitz operators and any bounded $L-$Toeplitz operator, $T$, can be factored as $T = F(R) ^* G(R)$ for $F,G \in R^\infty _d$ which are bounded below. If $T \geq 0$, and $A(R) ^* A(R) = I +T$, one can choose 
$$ F(R) := R_1 A(R) + R_2, \quad \mbox{and} \quad G(R) = R_1 A(R) - R_2, $$ so that 
$$ F(R) ^* F(R) = G(R) ^* G(R) = 2I + T \geq 2 I, $$ and 
$$ F(R) ^* G(R) = I + T - I = T \geq 0. $$ 
\end{lemma}
\begin{proof}{ (of Corollary \ref{ACfun})}
If $\mu \in (\A ^\dag ) _+$ is weak$-*$ continuous, then by \cite[Theorem 2.10]{DP-inv}, it is given by a vector state on the Fock space, $\mu = m_{x,y}$ for $x,y \in F^2 _d$. Alternatively, if $\mu$ is weak$-*$ continuous, then by the GNS representation, 
$$ \mu (L^\alpha ) = \ip{I + N _\mu}{ \Pi _\mu ^\alpha (I +  N _\mu )} _\mu, $$ so that by Definition \ref{WCdef}, $I + N _\mu$ is a weak$-*$ continuous vector for $\Pi _\mu$. Theorem \ref{DLP-intertwiner} then implies that there is a bounded intertwiner, $X : F^2 _d \rightarrow F^2 _d (\mu)$ and a vector $y \in F^2 _d$ so that 
$$ I + N _\mu = X y. $$ Hence, 
\ba \mu (a) & = & \ip{Xy}{\pi _\mu (a) Xy} _{\mu} \nn \\
& = & \ip{y}{X^* X a(L) y}_{F^2 _d}, \nn \ea and since $X L^\alpha = \Pi _\mu ^\alpha X$ is an intertwiner, $X^*X =T \geq 0$ is a bounded positive semi-definite $L-$Toeplitz operator. By Lemma \ref{MK-factor}, there are $F, G \in R^\infty _d$ so that $F(R) ^* G(R) = X^*X$. Setting $f := F(R) y$ and $g:=G(R) y$, we obtain
$$ \mu ( a ) = \ip{f}{a(L) g} _{F^2 _d} =m_{f,g} (a), $$ and $\mu$ is a vector state on the Fock Space. Conversely any positive vector state on the Fock Space is clearly weak$-*$ continuous.

If $\mu = m_{x,y}$ is weak$-*$ continuous (and positive), it is clear that the map $\Pi _k \mapsto L_k$ extends to a weak$-*$ homeomorphism since this is a WOT and hence weak$-*$ continuous functional on $\mc{L} (F^2 _d)$. Hence $\Pi _\mu$ is weak$-*$ continuous. 

If $\Pi _\mu$ is weak$-*$ continuous, Theorem \ref{ACrowiso} implies that $F^2 _d (\mu ) = WC (\Pi _\mu )$ so that every $h \in F^2 _d (\mu )$ is weak$-*$ continuous for $\Pi _\mu$. In particular, since $I + N _\mu$ is a weak$-*$ continuous vector for $\Pi _\mu$, we can repeat the above argument to show that $\mu = m_{f,g}$ is a vector state, hence weak$-*$ continuous.
\end{proof}

\begin{lemma} \label{domAC}
The positive cone of all weak$-*$ continuous NC measures $\La \in (\A ^\dag ) _+$ is \emph{hereditary}: If $\la, \La \in (\A ^\dag ) _+$, $\La $ is weak$-*$ continuous and $\la \leq \La$ then $\la $ is also weak$-*$continuous
\end{lemma}
\begin{proof}
If $\La$ is weak$-*$ continuous, then by Theorem \ref{DLP-intertwiner}, there is an intertwiner $X$ and a $y \in F^2 _d$ so that $X y = 1 \in F^2 _d (\La )$ and
$\La (a) = \ip{y}{X^* X a(L) y } _{F^2} = \ip{Xy}{\pi _\La (a) Xy} _{\La }$. Now, assuming that $\la \leq \La$, there is a positive $\La$-Toeplitz contraction $D$ (\emph{i.e.} $\pi _\La (L_k) ^* D \pi _\La (L_j ) = \delta _{k,j} D$) so that $$ \la (a) = \ip{1}{D \pi _{\La } (a) 1} _{\La} = \ip{Xy}{D \pi _\La (a) Xy} _{\La} = \ip{y}{X^* D X a(L) y } _{F^2}. $$
Since $D$ is $\La$-Toeplitz and $X$ is an intertwiner, $X^* D X$ is $L$-Toeplitz, and by Lemma \ref{MK-factor}, $X^* D X = X(R) ^* Y(R)$ for some $X(R), Y(R) \in R^\infty _d$. It follows that $\la =m_{f,g}$ is also a vector state, with $f = X(R) y$, $g= Y(R) y$, so that it is also weak$-*$ continuous.
\end{proof}

\begin{remark}
A natural question is whether any positive weak$-*$ continuous functional $\mu = m_{x,y}$ on the Free Disk System necessarily has the symmetric form $\mu = m_h := m_{h,h}$ for some $h \in F ^2 _d$. We will say that any such positive weak$-*$ continuous non-commutative measure is \emph{asymmetric} if there is no $h \in F^2 _d$ so that $\mu = m_{x,y} = m_{h,h}$, and \emph{symmetric} if $x=y$, and we write $m_x = m_{x,x}$ in this case. It is a curious fact that if $\mu$ is of Cuntz type$-L$ then no such $h$ exists, so that $\mu$ is asymmetric, see Corollary \ref{ACasym}.
\end{remark}

\begin{thm} \label{symistypeL}
If $\mu = m_x$ is symmetric and weak$-*$ continuous, then $\mu$ is type$-L$. Assuming that $x = x(R)1$ where $x(R)$ is outer, the distance from $I + N_\mu$ to $F^2 _d (m_x) _0$ is $|x (0) |$.
\end{thm}
Recall here that $F^2 _d (m_x )_0$ denotes the closed linear span of the non-constant free monomials in $F^2 _d (m_x)$, see Equation (\ref{nconmon}). 

\begin{remark} \label{wlogouter}
There is no loss in generality in assuming that $x$ is outer. By Davidson-Pitts, any $x \in F^2 _d$ factors as $x= \Theta (R) y$, where $y \in F^2 _d$ is $L-$cyclic, \emph{i.e.} right-outer, and $\Theta (R) =  M^R _{\Theta ^\dag}$ is right-inner, \emph{i.e.} an isometry, so that for any $a_1, a_2 \in \A$,
\ba m_x (a_1 ^* a_2 ) & = & \ip{a_1(L) x}{a_2 (L) x}_{F^2} \nn \\
& = & \ip{a_1 (L) y}{\Theta (R) ^* \Theta (R) a_2 (L) y}_{F^2} \nn \\
& = & \ip{a_1 (L) y}{a_2 (L) y}_{F^2} = m_y (a_1 ^* a_2 ). \nn \ea 
\end{remark}
\begin{proof}
Define $U _x : F^2 _d (\mu ) \rightarrow F^2 _d$ by 
$$ U _x ( L^\alpha + N_\mu ) := L^\alpha x \in F^2 _d. $$ This is an isometry, which is onto since $x$ is $L-$cyclic (since $x(R)$ is outer). It follows that $U _x \Pi _\mu U _x ^* = L$, so that $\Pi _\mu$ is pure type$-L$, and hence $\Pi _\mu$ is not Cuntz.

However, we can say more: Consider,
$$  \Delta _x  :=  \inf _{p(0) =0} \| (I - p (L ) ) + N_\mu \| _\mu  ^2, $$ this is the distance (squared) from $I +N _\mu$ to $F^2 _d (\mu ) _0 = \bigvee _{\alpha \neq \emptyset} (L^\alpha + N _\mu ). $ 
Hence, $\Delta _x =0$ if and only if the distance from $I + N_\mu$ to $F^2 _d (\mu ) _0$ vanishes, \emph{i.e.} if and only if $\mu$ is column-extreme in the sense of \cite[Definition 6.1, Theorem 6.4]{JM-freeCE}. Then, calculating as in \cite[Theorem 1.3]{Pop-entropy},
\ba  \Delta _x & = & \inf _{p(0) =0}  \| (I - p (\Pi _\mu ) ) (I + N _\mu ) \| _\mu ^2  \nn \\
& = &  \inf _{p (0)  = 0 } \| (I - p(L)  ) x \| ^2 _{F^2 _d } \nn \\
& = & \inf _{\mbf{q} \in \C \{ \mf{z} _1 , .. \mf{z} _d \} \otimes \C ^d} 
\| x  - L \mbf{q}(L) x  \| ^2 _{F^2 _d} \nn \\
& = & \inf _{\mbf{y} \in F^2 _d \otimes \C ^d } \| x - L \mbf{y} \| ^2 _{F^2 _d} \quad \quad \mbox{(Since $x$ is cyclic.)} \nn \\
& = & \| P _{\ran{L}} ^\perp x \| ^2 _{F^2 _d} \nn \\
& = & \| P _{\{1 \} } x \| _{F^2 _d} ^2 \nn \\
& = & |x (0) | ^2. \nn \ea 
\end{proof}

\begin{cor} \label{ACasym}
    If $\mu \in ( \A ^\dag ) _+$ is column-extreme (\emph{i.e.} $\Pi _\mu$ is Cuntz) and weak$-*$ continuous, then there is no $x \in F^2 _d$ so that $\mu = m_x$.
\end{cor}

There are many examples of absolutely continuous and column-extreme $\mu \in (\A ^\dag ) _+$, see, \emph{e.g.} \cite[Example 2.11]{DLP-ncld}. (This provides an example of a cyclic and absolutely continuous Cuntz row isometry, which is therefore not unitarily equivalent to copies of the left free shift. The fact that it is cyclic implies that it is unitarily equivalent to the GNS row isometry of a Cuntz type$-L$ NC measure.)
\begin{cor} \label{symchar}
$\Pi _\mu$ is of pure type$-L$ if and only if $\mu = m_x$ is symmetric and weak$-*$ continuous.
\end{cor}
\begin{proof}
One direction is in the proof of the previous theorem, Theorem \ref{symistypeL}. Namely if $\mu = m_x$ then $\Pi _\mu$ is of type$-L$. 

Conversely if $\Pi _\mu$ is type$-L$, then $\Pi _\mu$ is unitarily equivalent to copies of $L$. But, since $\Pi _\mu$ has a cyclic vector, it is unitarily equivalent to $L$. If $U : F^2 _d \rightarrow F^2 _d (\mu )$ is the unitary so that $U L_k = \pi _\mu (L_k ) U$, then, choosing $h \in F^2 _d$ so that $Uh = I + N_\mu$ yields:
\ba \mu (L^\alpha ) & = & \ip{I+N\mu}{\Pi _\mu ^\alpha (I+ N _\mu )}_\mu \nn \\
& = & \ip{Uh}{\Pi _\mu ^\alpha Uh} _\mu \nn \\
& = & \ip{h}{L^\alpha h}_{F^2 _d} = m_h (L^\alpha ). \nn \ea
\end{proof}

\subsection{Type$-L$ NC measures: The Helson-Lowdenslager approach} \label{ACsect}

Given an NC measure $\mu \in (\A ^\dag ) _+$, let $P_0$ denote the orthogonal projection of $F^2 _d (\mu )$ onto 
$$ F^2 _d (\mu ) _0 = \bigvee _{\alpha \neq \emptyset} (L^\alpha + N_\mu). $$ The following lemma is motivated by \cite[Chapter 4, Section 1]{Hoff}:
\begin{lemma}
There is a constant $c^2 \geq 0$ so that 
$$ c^2  m (L^\alpha ) = \ip{(I-P_0) (I + N_\mu ) }{\Pi _\mu ^\alpha (I- P_0 ) (I+N _\mu )} _\mu, $$ where $m$ is (normalized) NC Lebesgue measure.
\end{lemma}
\begin{proof}
This follows immediately from the fact that $F^2 _d (\mu ) _0$ is $\Pi _\mu -$invariant so that 
$$ \ip{(I-P_0) (I + N_\mu )}{\Pi _\mu ^\alpha (I- P_0 ) (I + N_\mu )} _\mu  =  \| (I - P_0 ) (I + N _\mu ) \| ^2 _\mu \delta _{\alpha , \emptyset } =   c ^2 m (L^\alpha), $$ with $c = \| (I - P_0 ) (I + N_\mu ) \| $.
\end{proof}
Define the co-isometry $W : F^2 _d (\mu ) \rightarrow F^2 _d (m ) = F^2 _d$ with initial space 
$$ \ker{W} ^\perp = \bigvee _{\alpha} \Pi _\mu ^\alpha (I - P_0 ) (I + N _\mu ), $$ 
\be W \Pi _\mu ^\alpha P_0 ^\perp (I + N_\mu ) = c L^\alpha + N_\mu. \ee 
\begin{prop}
The vector $P_0 ^\perp (I + N_\mu )$ is wandering for $\Pi _\mu$ so that 
$$ \ker{W} ^\perp = \bigoplus \{ \Pi _\mu ^\alpha P_0 ^\perp (I +N_\mu) \}. $$ The subspace $\ker{W} ^\perp$ is $\Pi _\mu-$reducing, the restriction
of $\Pi _\mu$ to $\ker{W} ^\perp $ is unitarily equivalent to $L$, and $W^* W = P_L$, the projection onto the type$-L$ part of $\mu$.
\end{prop}
\begin{proof}
The vector $w:= P_0 ^\perp (I + N_\mu ) $ is wandering since,
$$ \ip{ \Pi _\mu ^\alpha P_0 ^\perp (I + N_\mu ) }{\Pi _\mu ^\beta P_0 ^\perp (I + N_\mu )}  = \delta _{\alpha ,\beta} c^2. $$ The subspace $\ker{W} ^\perp$ is $\Pi _\mu -$invariant, by construction. Suppose that $h \in \ker{W}$, so that for any $\alpha \in \F ^d$, 
$$ 0 = \ip{ h }{\Pi _\mu ^\alpha (I - P_0 ) (I + N_\mu ) } _\mu. $$ For any $\alpha \neq \emptyset$,
$$ \ip{h}{\pi _\mu (L ^\alpha ) ^* (I - P_0 ) (I + N_\mu ) } _\mu = \ip{(I-P_0) \Pi _\mu ^\alpha h}{I+N_\mu} _\mu =0, $$
since $\Pi _\mu (L^\alpha ) h \in F^2 _d (\mu ) _0$ for any $\alpha \neq \emptyset$. Since $h \in \ker{W}$ was arbitrary it follows that 
$$ \ker{W} ^\perp = \bigvee \pi _\mu (\A + \A ^* ) (I -P_0) (I + N _\mu ), $$ is $\Pi _\mu -$reducing.

Since $W^*W$ is reducing for $\Pi _\mu$ and generated by the wandering vector $P_0 ^\perp (I + N_\mu)$, it follows that $W^* W \leq P_L$. However the vector $I +N_\mu$ is cyclic for $\Pi _\mu$ so that $P_L (I + N_\mu )$ is also cyclic for the type$-L$ row isometry $\Pi _L$, and hence the wandering space of $\Pi _L$ is one-dimensional. Since $P_0 ^\perp (I +N_\mu) \in \ran{W^* W}$ is wandering for $\Pi _\mu$, it spans the wandering space for $\Pi _L$, and we obtain that $ \ker{W} ^\perp = F^2 _d (\mu _L ). $
\end{proof}

\section{Weak$-*$ vs. Absolute Continuity} \label{WCvsAC}

In this section we prove that any weak$-*$ continuous NC measure is an absolutely continuous NC measure. 

\subsection{NC measures dominated by NC Lebesgue measure} 

\begin{prop}
Suppose that $\mu \in (\A ^\dag ) _+$ is dominated by $m$, $\mu \leq t^2 m$. Then $\mu$ is both AC and weak$-*$ continuous. \\

If $E _\mu = (\mc{C} ^\mu) ^* \mr{e} _\mu ^* \mc{C} _m : F^2 _d \rightarrow F^2 _d (\mu )$, then $E _\mu$ is a bounded intertwiner with dense range (and norm at most $t$), 
$E _\mu L_k = \Pi ^\mu _k E _\mu$.
\end{prop}

\begin{proof}
If $\mu$ is dominated by $m$, then it is weak$-*$ continuous since the positive cone $WC(\A ^\dag ) _+$ is hereditary by Lemma \ref{domAC}. It is absolutely continuous, by definition since $\scr{H} ^+ (H _\mu ) \subset H^2 (\B ^d _\N )$ by Theorem \ref{NCdom}. The statement about the intertwiner $E_\mu$ follows immediately from Lemma \ref{embeddings}.
\end{proof}

An arbitrary weak$-*$ continuous NC measure $\mu \in WC(\A ^\dag ) _+$ is generally not dominated by $m$, and it is natural to ask whether the previous Cauchy Transform intertwining results can be extended to this general case. This is possible, if one allows for unbounded intertwiners:
\begin{defn}
Let $\Pi$ be a row isometry on a Hilbert space $\H$. A closed, operator $X : \dom{X} \rightarrow \H$, with dense domain in $F^2 _d$, is called an \emph{intertwiner} if $\dom{X}$ is $L-$invariant and 
$$ X L_k x = \Pi _k X x; \quad \quad x \in \dom{X}. $$
\end{defn}
\begin{lemma} \label{ubtwine}
Let $\Pi$ be a row isometry on a Hilbert space $\H$, and let $X : \dom{X} \rightarrow \H$ be a closed, densely-defined intertwiner, $\dom{X} \subseteq F^2 _d$. Any vector $y \in \ran{X} \cap \dom{X ^*} $ is a weak$-*$ continuous vector for $\Pi$.
\end{lemma}
\begin{proof}
If $X$ is densely-defined and closed, then its adjoint, $X^*$ is also densely-defined and closed, and $X^*X$ is densely-defined, closed, and positive semi-definite. Furthermore, $\dom{X^*X}  \subseteq \dom{X}$ is a core for $X$ (hence dense in $F^2 _d$). If $y \in \dom{X^* X}$ then $Xy \in \dom{X^*} \cap \ran{X}$, and 
$$ \ip{Xy}{\Pi ^\alpha Xy} _\H = \ip{X^*Xy}{L^\alpha y} _{F^2 _d}, $$ is a weak$-*$ continuous functional so that $Xy$ is a weak$-*$ continuous vector, by definition.
\end{proof}

\subsection{Symmetric AC functionals}

Before tackling the fully general case of an asymmetric weak$-*$ continuous NC measure, first suppose that $\mu = m_x = m_{x,x}$ is a symmetric positive weak$-*$ continuous functional, where $x \in F^2 _d$. The results of \cite{JM-F2Smirnov,JM-freeSmirnov} show that one can define $x(R)$, where $x(R)1 = x$ as a densely-defined, closed, and potentially unbounded right multiplier on the Fock space with symbol in the (right) \emph{Free Smirnov Class} $\scr{N} _d ^+ (R)$, the set of all ratios of bounded right multipliers, $B(R) A(R) ^{-1}$, with outer (dense range) denominator.  We will write $x(R) \sim R^\infty _d$ to denote that $x(R)$ is an unbounded right multiplier affiliated to the right free analytic Toeplitz algebra $R^\infty _d$ (\emph{i.e.} it commutes with the left free shifts). The (potentially unbounded) left-Toeplitz operator $T := x(R) ^* x(R)$ is then well-defined, closed, positive semi-definite and densely-defined. 

Given $x(R) \sim R^\infty _d$, there is an essentially unique choice of $A,B \in [R^\infty _d ] _1$ so that $x(R) = B(R) A(R) ^{-1}$, and if $\Theta _x (R)$ denotes the two-component column with entries $A,B$, then $\Theta _x (R)$ is an isometric right multiplier (right-inner) from one to two copies of the NC Hardy space and $\ran{\Theta _x (R)} = G(x(R))$, the graph of $x(R)$ \cite[Corollary 4.27, Corollary 5.2]{JM-freeSmirnov}. Moreover, $x = x(R) 1$ belongs to $F^2 _d$ if and only if $A^{-1} := A(R) ^{-1} 1 \in F^2 _d$. In this case $L^\infty _d 1 \subseteq \dom{x(R)}$ \cite[Lemma 5.3]{JM-freeSmirnov}. We can further assume that $\fp$ is a core for $x(R)$ (if not, define $\check{x} (R)$ as the closure of $x(R)$ restricted to $\fp$), so that for any $y \in \dom{x(R)}$, there are free polynomials $p_n \in \fp$ so that $p_n \rightarrow y$ and $x(R) p_n \rightarrow x(R) y$. Recall, by Remark \ref{wlogouter}, we can assume without loss in generality that $x$ is outer, \emph{i.e.} $L-$cyclic, or equivalently, $x(R)$ has dense range. Let $U_x : F^2 _d (\mu ) \rightarrow F^2 _d$ be defined by 
$$ U_x (L^\alpha + N _\mu ) = L^\alpha x. $$ This is clearly an isometry, and since $x$ is assumed to be outer, it is onto $F^2 _d$.
\begin{thm} \label{CTF2}
Let $\mu = m_x \in WC (\A ^\dag ) _+$ be a symmetric weak$-*$ continuous NC measure, where $x \in F^2 _d$ is outer. A vector $y _\mu \in F^2 _d ( \mu )$ is such that $\mc{C} _\mu  y_\mu \in H^2 (\B ^d _\N )$ if and only if $U_x y_\mu =: y \in F^2 _d$ belongs to $\dom{x(R) ^*}$. 
\end{thm}
Since $U_x$ and $\mc{C} _\mu$ are unitary and $\dom{x(R)^*}$ is dense, it follows that if $\mu = m_x$ is symmetric and weak$-*$ continuous then  
$$ \mc{C} _\mu U_x ^* \dom{x(R)^*} = H^2 (\B ^d _\N ) \bigcap \scr{H} ^+ ( H_\mu ), $$ is dense in $\scr{H} ^+ (H_\mu )$ so that $\mu = m_x \in AC (\A ^\dag ) _+$ is an absolutely continuous NC measure.
\begin{proof}
Suppose that $y \in \dom{x(R) ^*}$, and consider $\mc{C} _\mu U_x ^* y  \in \scr{H} ^+ ( H _\mu )$. Then, 
\ba \left( \mc{C} _\mu U_x ^* y  \right) (Z) & = & \sum _{\alpha} Z^\alpha \ip{\Pi _\mu ^\alpha (I + N_\mu ) }{U_x ^* y} _\mu \nn \\
& = & \sum Z^\alpha \ip{U_x (L^\alpha + N_\mu) }{y}_{F^2 _d} \nn \\
& = & \sum Z^\alpha \ip{L^\alpha x}{y}_{F^2 _d} \nn \\
& = & \sum Z^\alpha \ip{L^\alpha 1}{x(R) ^* y} _{F^2 _d}. \nn \ea 
This shows that $\mc{C} _\mu U_x ^* y$ has the same Taylor-MacLaurin coefficients as $x(R) ^* y \in F^2 _d$, and hence belongs to $H^2 (\B ^d _\N )$.

Conversely, suppose that $y_\mu \in F^2 _d (\mu )$ is such that $h:= \mc{C} _\mu y_\mu$ belongs to $H^2 (\B ^d _{\N} )$.  Then, setting $y = U_x y_\mu$, and $h:= \mc{C} _\mu y _\mu \in H^2 (\B ^d _\N )$, 
\ba h (Z) & = & \sum _{\alpha} Z^\alpha \ip{\Pi _\mu ^\alpha (I + N_\mu ) }{ y_\mu } _\mu \nn \\
& = & \sum Z^\alpha \ip{U_x (L^\alpha + N_\mu) }{U_x y _\mu }_{F^2 _d} \nn \\
& = & \sum Z^\alpha \ip{L^\alpha x}{y}_{F^2 _d} \nn \\
& = & \sum Z^\alpha \ip{x(R) L^\alpha 1}{y} _{F^2 _d}. \nn \ea 
Identifying $h$ with an element of $F^2 _d$, the Fourier coefficients of $h$ are:
$$ h_\alpha := \ip{L^\alpha 1}{h} = \ip{x(R) L^\alpha 1}{y} _{F^2 _d}, $$ and it follows that 
for any $p \in \fp$, 
$$ \ip{p(L) 1}{h}_{F^2 _d} = \ip{x(R) p(L) 1}{y}_{F^2 _d}. $$ Since free polynomials are a core for $x(R)$, this proves that $y \in \dom{x(R) ^*}$ and that $x(R) ^* y = h$. 
\end{proof}

\begin{cor} \label{embeddense}
If $\mu = m_x$ is a symmetric weak$-*$ continuous NC measure, the intersection space $$\scr{H} ^+ (H _\mu) \bigcap H^2 (\B ^d _\N ) =: \dom{ \mr{e} _\mu } $$ is dense in $\scr{H} ^+ (H _\mu )$ and the embedding $\mr{e} _\mu : \dom{\mr{e} _\mu} \hookrightarrow H^2 (\B ^d _\N )$ is densely-defined and closed. That is, any symmetric weak$-*$ continuous NC measure is absolutely continuous. 
\end{cor}
\begin{proof}
The domain of $\mr{e} _\mu$ is dense by the previous Proposition. It remains to show that $\mr{e} _\mu$ is closed. If $f_n \rightarrow f$ in $\scr{H} ^+  (H _\mu)$ and $\mr{e} _\mu f_n \rightarrow g$ in $F^2 _d$, then in particular, $f_n (Z) = (\mr{e} _\mu f_n) (Z) \rightarrow g(Z)$ for $g \in F^2 _d$ and also $f_n (Z) \rightarrow f(Z)$ so that $f(Z) = g(Z)$ and $f \in \scr{H} ^+  (H _\mu ) \bigcap H^2 (\B ^d _\N ) = \dom{\mr{e} _\mu}$. 
\end{proof}
\begin{cor} \label{intertwinesym}
    The unbounded operator $X _\mu := (\mc{C} _\mu) ^* \mr{e} _\mu ^* : H^2 (\B ^d _\N ) \rightarrow F^2 _d (\mu ) $ is a closed, unbounded intertwiner with dense range and every vector in the dense set $\dom{X_\mu ^*} \cap \ran{X _\mu }$ is a weak$-*$ continuous vector for $\mu$.
    Equivalently, the embedding $E_\mu = X _\mu \mc{C} _m : F^2 _d (m) \rightarrow F^2 _d (\mu )$ is closed, densely-defined and has dense range. 
\end{cor}
\begin{lemma} \label{denserangelem}
Let $T$ be a closed, densely-defined linear operator on $\dom{T} \subseteq \H$. If $\ran{T}$ is dense, then $\dom{T^*} \cap \ran{T}$ is dense and contains the dense linear space $\ran{T (I+T^*T) ^{-1}}$. 
\end{lemma}
\begin{proof}
Set $\Delta _T := (I +T ^* T ) ^{-1}$, this is a strictly positive contraction \cite[Theorem 5.19]{AnalysisNow}. Moreover, $\ran{\Delta _T}$ is a core for $T$, so that the set of all pairs 
$ (x, Tx)$, for $x \in \ran{\Delta _T}$ is dense in the graph of $T$. In particular, given any $Ty \in \ran{T}$, one can find $(x _n , Tx_n ) $ with $x_n \in \ran{\Delta _T}$ so that $x_n \rightarrow y$ and $Tx_n \rightarrow Tx$. Since we assume that $\ran{T}$ is dense it follows that $T \Delta _T = T (I + T^*T) ^{-1}$ also has dense range. Moreover, again by \cite[Theorem 5.19]{AnalysisNow}, $T \Delta _T$ is a contraction and $\ran{\Delta _T} = \dom{T^*T}$ so that $\ran{T \Delta _T} \subset \dom{T^*}$. In conclusion $\ran{T \Delta _T} \subseteq \ran{T} \cap \dom{T^*}$ is dense.
\end{proof}
\begin{proof}{ (of Corollary \ref{intertwinesym})}
The proof goes through as in the case where $\mu$ is dominated by $m$, using that $X$ is closed operator, as in Lemma \ref{ubtwine}. In particular, $H^2 (\B ^d _\N ) \bigcap \scr{H} ^+ (H _\mu ) = \dom{\mr{e} _\mu }$ is dense, and $\mr{e} _\mu$ is by definition injective on its domain, and closed by the previous corollary. It follows that $\mr{e} _\mu ^*$ is also closed, densely-defined, and has dense range, so that $\ran{X _\mu }$ is also dense in $F^2 _d (\mu ).$ Since $X _\mu$ is a closed operator with dense range, the previous general lemma shows that $\dom{X_\mu ^*} \cap \ran{X_\mu}$ is dense. Lemma \ref{ubtwine} now implies that every vector in this dense set is a weak$-*$ continuous vector.
\end{proof}

\subsection{Asymmetric AC functionals}

Even more generally, suppose that $\mu \in WC (\A ^\dag ) _+$ is an arbitrary weak$-*$ continuous NC measure. By Corollary \ref{ACfun}, $\mu = m_{x,y} = m_{y,x} \geq 0$ is a vector state on the Fock space with $x,y \in F^2 _d$.

\begin{lemma}
Any $\mu \in WC (\A ^\dag ) _+$ has the form:
$$ \mu ( L^\alpha ) = \ip{h}{\tau L^\alpha h}_{F^2 _d}, $$ where $h$ is outer, \emph{i.e.} $L-$cyclic, and $\tau \geq 0$ is a bounded, positive semi-definite $L-$Toeplitz operator.
\end{lemma}
\begin{proof}
This is as in the proof of Corollary \ref{ACfun}. Since $\mu$ is weak$-*$ continuous, every vector in $F^2 _d (\mu )$ is a weak$-*$ continuous vector. In particular, there is a $g \in F^2 _d (\mu )$, and a bounded intertwiner $X : F^2 _d \rightarrow F^2 _d (\mu )$ so that 
$$ X g = I + N _\mu. $$ Since $g \in F^2 _d$, $g = g(R) 1$, where $g(R) \sim R^\infty _d$ is an unbounded right multiplier, and $g(R)$ has the Smirnov factorization $g(R) = N(R) D(R) ^{-1}$, where $N, D \in [R^\infty _d ] _1$, and $D$ is outer. If $\Theta  (R) F (R)$ is the inner-outer factorization of $N(R)$, set $h := F (R) D(R) ^{-1} 1 \in F^2 _d$, and $\tau := \Theta (R) ^* X^* X \Theta (R) \geq 0$, a bounded, positive semi-definite $L-$Toeplitz operator. Then,
\ba \mu ( L^\alpha ) & = & \ip{I + N _\mu}{\Pi _\mu ^\alpha (I + N _\mu)}_{\mu} \nn \\
& = & \ip{Xg}{\Pi _\mu ^\alpha Xg}_{\mu} \nn \\
& = & \ip{\Theta (R) h}{X^*X L^\alpha \Theta (R) h } _{F^2 _d} \nn \\
& = & \ip{h}{\tau L^\alpha h}_{F^2 _d}. \nn \ea
\end{proof}

For any $\eps > 0$, define $\mu _\eps \in WC (\A ^\dag ) _+$ by 
\be \mu _\eps (L^\alpha ) := \ip{h}{(\tau +\eps I ) L^\alpha h}_{F^2 _d}, \label{mueps} \ee Since $\tau +\eps I$ is bounded below, Theorem \ref{Popfactor} implies that it is factorizable:
$$ \tau+\eps I = A_\eps (R) ^* A_\eps (R), $$ for some outer $A_\eps (R) \in R^\infty _d$. Hence, setting $g _\eps := A_\eps (R) h \in F^2 _d$, $\mu _\eps = m_{g _\eps}$ is a symmetric vector state, so that $\mu _\eps$ is absolutely continuous for any $\eps >0$ by Theorem \ref{embeddense}.

\begin{prop} \label{monSRprop}
Let $T_\eps$ be the closed, positive semi-definite $L-$Toeplitz operator so that $q_{T_\eps}$ is the closure of the form generated by $\mu _\eps$. Then $T_\eps$ is convergent in the strong resolvent sense to a closed, positive semi-definite $L-$Toeplitz $T$, where $q_T $ is the closure of the absolutely continuous part of $q_\mu$.
\end{prop}
This proposition is a straightforward consequence of the monotone convergence theorem for decreasing nets of positive semi-definite quadratic forms, due to B. Simon \cite[Theorem 3.2]{Simon1}. Recall here that a sequence of closed, positive semi-definite operators $T_n$ is said to converge to a closed, positive semi-definite operator $T \geq 0$ in the strong resolvent (SR) sense, if 
$$ (I + T_n ) ^{-1} \stackrel{SOT}{\rightarrow} (I + T) ^{-1}, $$ where $SOT$ denotes the strong operator topology \cite[Chapter VIII.7]{RnS1}.
\begin{proof}
Observe that the positive semi-definite forms $q_\eps := q_{T_\eps}$ all have the free polynomials, $\fp$, as a common form core, that 
$$ q_\eps (p,q ) \rightarrow q_\mu (p, q), $$ as $\eps \downarrow 0$,
and that the $q_\eps := q_{T_\eps}$ are monotonically decreasing as $\eps \downarrow 0$. The proposition statement is now an immediate consequence of \cite[Theorem 3.2]{Simon1} (see also \cite[Theorem S.16]{RnS1})  
\end{proof}

Our goal now is to show that $\mu = \mu _{ac}$ is absolutely continuous by showing that $q_T$ is the closure of $q_\mu$. The strategy is to `peel off' the adjunction by $h(R)$ and its adjoint from $T_\eps +I$, and to consider the invertible, positive operators: 
\be S_\eps := (h(R) ^{-1} ) ^* h(R) ^{-1} + \tau + \eps I; \quad \eps \geq 0, \label{Seps} \ee with common domain
$$ \dom{S_\eps} = \dom{ (h(R) ^{-1}) ^* h(R) ^{-1}} = \ran{ h(R) h(R) ^*}. $$ (Given any closed, self-adjoint operator $S$, and a bounded self-adjoint operator $A$, it is straightforward to verify that $S +A$ is closed, and self-adjoint on $\dom{S}$.)
Since each of the $S_\eps$ is invertible, the quadratic forms of their inverses are a monotonically increasing net of positive quadratic forms, and we can then apply B. Simon's second monotone convergence theorem for quadratic forms to conclude, ultimately, that $\ov{q_\mu} = q_T$.

For any $\eps \geq 0$ consider the positive quadratic form $Q_\eps := Q_{S_\eps}$:
\begin{align*} Q_\eps (x, x) & := \ip{ h(R) ^{-1} x }{h(R) ^{-1} x} _{F^2 _d} + \ip{x}{(\tau + \eps I ) x} _{F^2 _d} \\
 x \in \mc{D} & = \dom{h(R) ^{-1} } = \ran{h(R)}, \end{align*} where $h$ is as above, in Equation (\ref{mueps}). This is well-defined since $h(R)$ is outer, where $h = h(R) 1$ (note that $h(R) ^{-1}$ is also outer). Further observe that
$$ \dom{S_\eps ^{1/2} } = \dom{h(R) ^{-1} } = \ran{h(R)}, $$ for every $\eps \geq 0$ and that $S_\eps$ is bounded below by $\eps I$. 

\begin{lemma} \label{SnotSR}
    The strictly positive $L-$Toeplitz operators $S_\eps$ converge in the strong resolvent sense to $$ S_0 = (h(R) ^{-1} ) ^* h(R) ^{-1} + \tau \geq 0. $$ 
\end{lemma}
\begin{proof}
    Since all of the $S_\eps$ have the same domain for $\eps \geq 0$, fix any $x \in \dom{S_\eps} = \dom{S_0} = \dom{ (h(R) ^{-1}) ^* h(R) ^{-1}}, $ and observe that
$$ S_\eps x = (h(R) ^{-1}) ^* h(R) ^{-1} x + (\tau + \eps I ) x, $$ which clearly converges to $S_0 x$ as $\eps \downarrow 0$. By \cite[Theorem VIII.25 (a)]{RnS1}, $S_\eps$ converges to $S_0$ in the strong resolvent sense.
\end{proof}

\begin{lemma}
For any $\eps >0$, the operator $S_\eps ^{1/2} h(R)$ is closed on $\dom{h(R)}$, and $(S_\eps ^{1/2} h(R) ) ^* = h(R) ^* S_\eps ^{1/2}$.
\end{lemma}
\begin{proof}
First, $\ran{h(R)} = \dom{S_\eps ^{1/2}}$ so that $S_\eps ^{1/2} h(R)$ is densely-defined. If $h(R) ^{-1} x_n \rightarrow y$, and $S_\eps ^{1/2} h(R) h(R) ^{-1} x_n \rightarrow g$, then $x_n \rightarrow x$ is convergent since $S_\eps \geq \eps I$ is bounded below. Since $h(R) ^{-1}$ is closed on $\ran{h(R)}$, it follows that $x \in \dom{h(R) ^{-1}} = \ran{h(R)}$ and $h(R) ^{-1} x = y$. Also $S_\eps ^{1/2}$ is closed so that $S_\eps ^{1/2} x_n \rightarrow S_\eps ^{1/2} x  =g$. Since $x \in \dom{h^{-1}}$ it then follows that 
$$ g= S_\eps ^{1/2} x = S_\eps ^{1/2} h(R) h(R) ^{-1} x = S_\eps ^{1/2} h(R) y, $$ proving that $S_\eps ^{1/2} h(R)$ is closed on this domain. 

To prove the second statement, fix $x \in \dom{(S_\eps ^{1/2} h(R) ) ^*}$ and consider any $y = h(R) ^{-1} g \in \dom{S_\eps ^{1/2} h(R)}$. Then,
$$ \ip{(S_\eps ^{1/2} h(R) ) ^* x}{y}  =  \ip{x}{S_\eps ^{1/2} g}, $$ holds for any $g \in \dom{h(R) ^{-1}}$, so that $x \in \dom{S_\eps ^{1/2}}$ and the above is equal to
$$ \ip{S_\eps ^{1/2} x}{g} = \ip{S_\eps ^{1/2} x}{h(R) y}. $$ Again, this holds for every $y \in \dom{h} = \ran{h^{-1}}$ so that $S_\eps ^{1/2} x \in \dom{h(R) ^*}$, and the above is equal to
$$ \ip{h(R) ^* S_\eps ^{1/2} x}{y}, $$ proving the second claim.
\end{proof}

\begin{lemma}
For any $\eps >0$, $T_\eps +I = h(R) ^* S_\eps h(R)$, and $\dom{T_\eps ^{1/2}} = \dom{h(R)}$.
\end{lemma}
\begin{proof}
The last statement is essentially by definition, $T _\eps = g_\eps (R) ^* g_\eps (R)$, where $g_\eps (R) := A_\eps (R) h(R)$, and $A_\eps (R) ^* A_\eps (R) = \tau + \eps I$ is a bounded, invertible operator. By polar decomposition, $\dom{\sqrt{T_\eps}} = \dom{g_\eps (R)} = \dom{h(R)}$.

Let $q_{\eps} + m := q_{T_\eps + I}$, and as before $Q_\eps := q_{S_\eps}$. Then for any $x \in \dom{T_\eps ^{1/2}} = \dom{h(R)} = \ran{h(R) ^{-1}} $, $x = h(R) ^{-1} x' $, we have that
\ba (q_{\eps} +m)(x,x) & = & \ip{h(R) x}{(\tau + \eps I) h(R) x}_{F^2 _d} +  \ip{x}{x}_{F^2 _d} \nn \\
& = & \ip{x'}{(\tau +\eps I) x'}_{F^2 _d} + \ip{h(R) ^{-1} x'}{h(R) ^{-1} x' } _{F^2 _d} \nn \\
& = & Q_\eps (x' , x' ) \nn \\
& = & \ip{ S_\eps ^{1/2} h(R) x }{ S_\eps ^{1/2} h(R) x}. \nn \ea 
It follows that the positive operators $T_\eps + I$ and $h(R) ^* S_\eps h(R)$ define the same closed quadratic form, and hence, by uniqueness (see \cite[Chapter VI, Theorems 2.1, 2.23]{Kato}) we have that 
$$ h(R) ^* S_\eps h(R) = T_\eps +I. $$ 
\end{proof}

Consider the bounded, postive quadratic forms:
$$ q_\eps ^{-1} := q_{(I+T_\eps ) ^{-1}}, \quad \mbox{and} \quad Q_\eps ^{-1} = q_{S_\eps ^{-1}}. $$  Since the $T_\eps \geq 0$ are monotonically decreasing as $\eps \downarrow 0$, a result of Kato \cite[Chapter VI, Theorem 2.21]{Kato} implies that the bounded operators $0 \leq (I + T_\eps ) ^{-1}$ are monotonically increasing as $\eps \downarrow 0$. Moreover, $(I +T_\eps ) ^{-1}$ is a contraction and $(I+T_\eps ) ^{-1}$ converges in SOT to $(I +T) ^{-1}$ as $\eps \downarrow 0$ by Proposition \ref{monSRprop}. Notice that $S_\eps$ is positive and invertible for every $\eps >0$, and positive and injective for $\eps =0$. Since $Q_\eps$ is monotonically decreasing, the net $Q_\eps ^{-1}$ is a monotonically increasing net of bounded (but not uniformly bounded) positive quadratic forms, and the second monotone convergence theorem of B. Simon applies:

\begin{thm}{ (\cite[Theorem S.14]{RnS1}, \cite[Theorem 3.1, Theorem 4.1]{Simon1})} \label{Simonup}
Let $(q_k)$ be a monotonically non-decreasing sequence of closed, positive semi-definite quadratic forms which are densely defined in a Hilbert space, $\H$. Let 
$$\dom{q_\infty} := \{ \left. x \in \bigcap \dom{q_k} \right| \ \sup q_k (x,x) < + \infty \}, $$
and set 
$$ q_\infty (x,y) := \lim _{n\rightarrow \infty} q_k (x,y); \quad \quad x,y \in \dom{q_\infty}.$$ Then $q_\infty$ is also positive semi-definite, and closed on $\dom{q_\infty}$. If $q_\infty$ is densely-defined and if $T_k, T_\infty$ are the closed, densely-defined and positive semi-definite operators so that  $q_k = q_{T_k}$, $q_\infty = q_{T_\infty}$, then $T_k$ converges to $T_\infty$ in the strong resolvent sense.
\end{thm}

\begin{cor}
The quadratic forms of $(I +T) ^{-1}$ and $h(R) ^{-1} S_0 ^{-1} (h(R) ^*) ^{-1}$ agree on free polynomials, and $q_\mu $ is closable so that $\mu \in AC (\A ^\dag ) _+$.
\end{cor}
\begin{proof}
We have shown that $(I+T_\eps) = h(R) ^* S_\eps h(R)$, for any $\eps >0$. Since $S_\eps ^{1/2} h(R)$ and $h(R) ^* S_\eps ^{1/2}$ are closed and bounded below by $1$ on their domains, it follows that $S_\eps ^{-1/2} (h(R))^{-1}) ^* $ is bounded, and extends by continuity to a contraction.
Given any free polynomial, $p \in \fp$,
\ba q_\eps ^{-1} (p, p ) & = & Q_\eps ^{-1} ( (h(R) ^{-1} ) ^* p , (h(R) ^{-1} ) ^* p ) \nn \\
& = & Q_\eps ^{-1} (p_h ,p_h), \quad \quad p_h := (h(R) ^{-1} ) ^* p \in \fp. \nn \ea
This remains bounded as $\eps \downarrow 0$, and,
$$ \mc{D} _0 := \bigvee (h(R) ) ^{-1} ) ^* \fp, $$ is dense in $F^2 _d$ since $h(R) ^{-1}$ is right-Smirnov, so that the free polynomials are a core for its adjoint, and $h(R) ^{-1}$ is injective so that its adjoint has dense range \cite[Corollary 3.13, Corollary 3.15, Remark 3.16]{JM-freeSmirnov}. The previous Theorem \ref{Simonup}, then implies that 
$$ Q_0 ^{-1}  (x,y) := \lim _{\eps \downarrow 0} Q_\eps ^{-1} (x,y), $$ is a closed, densely-defined, positive semi-definite quadratic form on some form domain $\dom{Q_0 ^{-1} } \supseteq \mc{D} _0$. Since $Q_0 ^{-1}$ is closed, it is the quadratic form of some closed $\wt{S} _0 ^{-1}$, and Theorem \ref{Simonup} implies that $S_\eps ^{-1}$ converges in the strong resolvent sense to $\wt{S} _0 ^{-1}$.  However, by Lemma \ref{SnotSR}, $S_\eps$ converges in the strong resolvent sense to $S_0$, where $S_0$ is injective so that $S_0 ^{-1}$ is densely-defined and closed. In particular,
$$ S_\eps (I + S_\eps ) ^{-1} = I - (I + S_\eps ) ^{-1} \stackrel{SOT}{\rightarrow} S_0 (I + S_0 ) ^{-1}. $$ However,
$$ S_{\eps} (I + S_\eps ) ^{-1} = S_\eps \left( (S_\eps ^{-1} +I ) S_\eps \right) ^{-1} =  (I + S_\eps ^{-1} ) ^{-1}, $$ for
any $\eps \geq 0$. It follows that $S_\eps ^{-1}$ converges in the strong resolvent sense to $S_0 ^{-1}$, so that $S_0 ^{-1} = \wt{S} _0 ^{-1}$. That is, $Q_0 ^{-1}$ is the quadratic form of $S_0 ^{-1}$. Hence, $p_h = (h(R) ^{-1} ) ^* p \in \dom{S_0 ^{-1/2}}$ for any $p \in \fp$, and \ba q_0 ^{-1} (p, p ) & = & q _{(I+T) ^{-1} } (p, p )  \nn \\
& = & Q_0 ^{-1} (p _h , p_h ) \nn \\
& = & q _{(S_0 ^{-1/2}) ^* S_0 ^{-1/2}} (p_h , p_h) \nn \\
& = & \ip{ S_0 ^{-1/2} (h(R) ^{-1}) ^* p }{S_0 ^{-1/2} (h(R) ^{-1}) ^* p} _{F^2 _d}. \nn \ea
Hence, $Y^* := \ov{S_0 ^{-1/2} (h(R)^{-1}) ^*} $ is a contraction so that $q_{YY^*} = q_{(I+T) ^{-1}}$. 

By polar decomposition, there is a unitary, $U$ so that $UY^* = \sqrt{I+T} ^{-1}$. Recall that,
$$ \dom{h(R)^*} = \ran{(h(R) ^* ) ^{-1}} = \dom{S_0 ^{-1/2}} = \ran{S_0 ^{1/2}}, $$ so that the 
operator 
$$ (Y^*) ^{-1} = \ov{h(R) ^* S_0 ^{1/2}}, $$ is well-defined, closed, and densely-defined, and 
$$ (Y^*) ^{-1} = \sqrt{I+T} U. $$ It follows that 
$$ q_{I+T} = q_{(Y^*) ^{-1} Y ^{-1}}, $$ so that for any $x \in \dom{h(R)}$, 
\ba q_T (x,x) + \ip{x}{x} _{F^2 _d} & = & q_{(Y^*) ^{-1} Y ^{-1}} (x,x) \nn \\
& = & q _{S_0} (h(R) x , h(R) x ) \nn \\
& = & q _{(h^{-1} ) ^* h^{-1}} (h(R) x,h (R) x) + q _\tau (h(R) x , h(R) x ) \nn \\
& = & \ip{x}{x} _{F^2 _d} + q_\mu (x,x). \nn \ea 
It follows that for all $x \in \dom{h(R)}$, 
$$ q_T (x,x) = q_\mu (x,x), $$ and since $q_T$ is closable, this proves that $q_\mu$, with domain $\dom{q_\mu} = \dom{h(R)} \supset \fp$ is a closable form.  By Theorem \ref{ACismaxclosed}, $\mu$ is an absolutely continuous AC measure.
\end{proof}

\begin{cor} \label{wcisac}
Any weak$-*$ continuous NC measure $\mu \in WC(\A ^\dag ) _+$ is an absolutely continuous NC measure, $WC (\A ^\dag ) _+ \subseteq AC (\A ^\dag ) _+$.
\end{cor}

\begin{remark}
The above result is in contrast to \cite[Theorem 4.4]{HdeSnoo}, which implies that if $q$ is any closable quadratic form which is densely-defined in a Hilbert space, $\H$, that either $q$ is bounded, or $q$ has a decomposition $q= q_1  +q_2$ where $q_1$ is again closable, and $q_2$ is singular. Since the positive cone of all weak$-*$ continuous NC measures is hereditary, if $q= q_\mu$ is not bounded, then $q_2$ cannot be the quadratic form of an NC measure, $\ga$, since $\ga$ would necessarily be weak$-*$ continuous so that $q_2$ would be a closable quadratic form by the above results. One can check that the decomposition in \cite[Theorem 4.4]{HdeSnoo} applied to $q_\mu$ can never yield $L-$Toeplitz forms $q_1$ and $q_2$. 

It was observed already in \cite[Section 2, Remark 2]{Simon1} that the set of all absolutely continuous (\emph{i.e.} closable) positive semi-definite quadratic forms with dense domain in a separable Hilbert space, is not hereditary. Namely, \cite[Section 2]{Simon1} provides an explicit example of a singular (positive semi-definite) quadratic form $q_1$, and a bounded positive semi-definite (hence absolutely continuous/ closable) quadratic form $q_2$ whose sum is absolutely continuous.  It is the extra $L-$Toeplitz structure of the quadratic forms we consider (\emph{i.e.} the fact that our quadratic forms correspond to NC measures) that ensures we obtain more precise analogues of Lebesgue decomposition theory.
\end{remark}

\section{The NC Lebesgue Decomposition}
\label{singsect}

\begin{thm} \label{aciswc}
If $\mu \in AC (\A ^\dag ) _+$ is absolutely continuous, then it is weak$-*$ continuous so that the positive cones of weak$-*$ continuous and absolutely continuous measures coincide.
\end{thm}

\begin{proof}
That $WC (\A ^\dag ) _+ \subseteq AC (\A ^\dag ) _+$ was proven in Corollary \ref{wcisac}.
If $\mu$ is absolutely continuous, then by definition the intersection space:
$$ \mr{int} (\mu , m ) := \scr{H} ^+ (H _\mu ) \bigcap H^2 (\B ^d _\N ), $$ is dense in $\scr{H} ^+ (H_\mu )$, and the embedding, $\mr{e} _{ac} : \mr{int} (\mu , m) \hookrightarrow H^2 (\B ^d _\N ) \simeq F^2 _d$ is densely-defined. As in the proof of Corollary \ref{embeddense}, it is straightforward to verify that $\mr{e} _{ac}$, with domain $\mr{int} (\mu , m)$ is closed. Notice also that $\mr{e} _{ac}$ is trivially a multiplier by the constant NC function $\mr{e} _{ac} (Z) = I_n$, for $Z \in \B ^d _n$. It follows that all of the kernel vectors $K \{Z , y , v \} $ belong to the domain of $\mr{e} _{ac} ^*$, and that 
$$ \mr{e} _{ac} ^* K \{ Z, y , v \} = K^{\mu} \{ Z , y , v \}. $$ It further follows that $\mr{e} _{ac} ^*$ intertwines $L$ and $V_\mu$: \ba \mr{e} _{ac} ^* L K_Z  Z^* &=& \mr{e} _{ac} ^* ( K  _Z - K _{0_n} ) \nn \\
& = & K^{\mu } _Z - K ^{\mu } _{0_n} =  (K^\mu _Z - K^\mu _{0_n} ) \nn \\
& = &  V_\mu K^\mu _Z Z ^* \nn \\
& = &  V_\mu \mr{e} _{ac} ^* K_Z Z^*. \nn \ea 
Since $\mr{e} _{ac}$ is closed and densely-defined, so is its adjoint, and it follows that $X := \mc{C} _\mu ^* \mr{e} _{ac} ^*$ is a closed, densely-defined intertwiner with dense range in $F^2 _d (\mu)$. By Lemma \ref{ubtwine} and Lemma \ref{denserangelem}, $\ran{X} \cap \dom{X ^*}$ is dense in $F^2 _d (\mu )$, and every vector in this set is a weak$-*$ continuous vector for $\mu$. Since $WC (\mu)$ is always closed, it follows that $F^2 _d (\mu ) = F^2 _d (\mu _{wc} )$ so that $\mu$ is a weak$-*$ continuous NC measure.
\end{proof}

\begin{defn}
A vector $x \in F^2 _d (\mu )$ is a \emph{weak$-*$ analytic vector} for $\Pi _\mu$ if the Free Cauchy Transform of $x$ belongs to $H^2 (\B ^d _\N )$.
\end{defn}

\begin{cor}
Any weak$-*$ analytic vector for $\Pi _\mu$ is a weak$-*$ continuous vector for $\Pi _\mu$, and the set of all weak$-*$ analytic vectors for $\Pi _\mu$ is dense in $F^2 _d (\mu _{ac} )$, the largest $\Pi _\mu -$reducing subspace of weak$-*$ continuous vectors for $\Pi _\mu$.
\end{cor}
\begin{proof}
This follows immediately from the proof of the previous theorem.
\end{proof}

\begin{thm} \label{sameLD}
A positive NC measure $\mu \in (\A ^\dag ) _+$ is weak$-*$ continuous if and only if it is absolutely continuous and weak$-*$ singular if and only if it is singular. In particular, if 
$$ \mu =\mu _{ac} + \mu _s = \mu _{wc} + \mu _{ws}, $$ are the Lebesgue Decomposition and weak$-*$ Lebesgue Decomposition of $\mu$, then $\mu _{ac} = \mu _{wc}$ and $\mu _{s} = \mu _{ws}$.

\end{thm}
\begin{proof}
Corollary \ref{wcisac} and Theorem \ref{aciswc} imply that $\mu$ is weak$-*$ continuous if and only if it is absolutely continuous. In particular given any $\mu \in (\A ^\dag ) _+$, $F^2 _d (\mu _{wc} )$ is the largest reducing subspace of weak$-*$ continuous vectors for $\Pi _\mu$, and the previous theorem shows that $F^2 _d (\mu _{ac} ) \subseteq F^2 _d (\mu _{wc} )$ so that $\mu _{ac} \leq \mu _{wc}$. Conversely, Corollary \ref{wcisac} shows that 
$$ \scr{H} ^+ (H _{\mu _{wc} } ) \bigcap H^2 (\B ^d _\N ), $$ is dense in $\scr{H} ^+ (H _{\mu _{wc} } )$ so that by definition, $\scr{H} ^+ (H _{\mu _{wc} } ) \subseteq \scr{H} ^+ (H _{\mu _{ac} } )$ and $\mu _{wc} \leq \mu _{ac}$. 

Comparing the two direct sum decompositions,
$$ \begin{array}{cccc} F^2 _d (\mu ) = & F^2 _d ( \mu _{ac} ) & \oplus  & F^2 _d (\mu _s ) \\
\verteq & \verteq & & \\
F^2 _d  (\mu ) = & F^2 _d (\mu _{wc} ) & \oplus & F^2 _d (\mu _{ws} ), \end{array} $$
shows that $F^2 _d (\mu _{s} ) = F^2 _d (\mu _{ws} )$, and we conclude that $\mu _{s} = \mu _{ws}$.
\end{proof}

The weak$-*$ Lebesgue Decomposition of any NC measure $\mu \in (\A ^\dag ) _+$ clearly recovers the classical Lebesgue decomposition of any finite, positive, regular Borel measure on the circle (with respect to normalized Lebesgue measure), in the single-variable case of $d=1$. Since the weak$-*$ Lebesgue decomposition and the Lebesgue Decomposition of any $\mu \in (\A ^\dag ) _+$ are the same by the above theorem, it follows that our reproducing kernel approach to Lebesgue decomposition theory provides a new proof of Lebesgue decomposition of positive measures on the circle:

\begin{cor} \label{classLD}
Let $\mu$ be a positive, finite, and regular Borel measure on the unit circle $\partial \D$. If 
$ \mu = \mu _{ac} + \mu _s,$ is the classical Lebesgue decomposition of $\mu$ into absolutely continuous and singular parts, then, 
$$ \scr{H} ^+ (H _\mu ) = \scr{H} ^+ (H _{\mu _{ac}}) \oplus \scr{H} ^+ (H _{\mu _s} ), $$ where $$ \scr{H} ^+ (H _{\mu _{ac}} ) = \left( \scr{H} ^+ (H _\mu ) \bigcap H^2 (\D ) \right) ^{-\| \cdot \| _{H_\mu}}, \quad \mbox{and} \quad \scr{H} ^+ (H _{\mu _s} ) \bigcap H^2 (\D ) = \{ 0 \}. $$ 
\end{cor}

\subsection{The cone of singular NC measures}

We have seen that $AC (\A ^\dag ) _+ = WC (\A ^\dag ) _+$ is a positive hereditary cone. It remains to show that $\mr{Sing} (\A ^\dag ) _+ = WS (\A ^\dag ) _+$ is also a positive cone (that it is hereditary was already proven in Lemma \ref{singhered}).

\begin{lemma} \label{WCisS}
If $\mu ,\la \in (\A ^\dag ) _+$, $\mu$ is singular and $\la$ is type$-L$, then 
$$ \scr{H} ^+ (H _\mu ) \bigcap \scr{H} ^+ (H _\la ) = \{ 0 \}. $$
\end{lemma}
In particular, by Theorem \ref{NCintersect}, this implies that that $F^2 _d (\mu + \la ) = F^2 _d (\mu ) \oplus F^2 _d (\la).$
\begin{proof}
Consider the closure of the intersection space in $\scr{H} ^+ (H _\mu )$:
$$ \mr{Int} _\mu (\la ):= \left( \scr{H} ^+ (H _\mu ) \bigcap \scr{H} ^+ (H_\la ) \right) ^{-\| \cdot \| _{H_\mu}}. $$ By Corollary \ref{symchar}, if $\la$ is pure type$-L$, then $\Pi _\la$ is pure type$-L$, \emph{i.e.} $\Pi _\la$ is unitarily equivalent to $L$ and hence has no direct summand of Cuntz type. By \cite[Theorem 6.4]{JM-freeCE}, $\la$ is not column-extreme and $\scr{H} ^+ (H _\la )$ then contains the constant functions so that Proposition \ref{reduceprop} applies. By Theorem \ref{redintersect} and Proposition \ref{reduceprop}, $\mr{Int} _\mu (\la)$ is closed and $V_\mu-$reducing, so the orthogonal projection $P_{\mu \cap \la} : \scr{H} ^+ (H_\mu ) \rightarrow \mr{Int} _\mu (\la )$ commutes with $V_\mu , V_\mu ^*$. Let 
$$ \mr{e}_{ac} : \mr{Int} _\mu  ( \la ) \hookrightarrow \scr{H} ^+ (H _\la ), $$ be the densely-defined embedding. As before (see the proof of Corollary \ref{embeddense}) it is easy to check that $\mr{e} _{ac}$ is closed on its maximal domain, $\dom{\mr{e} _{ac} } = \scr{H} ^+ ( H_\mu ) \bigcap \scr{H} ^+ (H _\la )$. Also as in the proof of Theorem \ref{aciswc}, since $\mr{e} _{ac}$ is trivially a multiplier by the constant NC function $\mr{e} _{ac} (Z) = I_n$, it follows that $\mr{e} _{ac} K^\la _\alpha = K^{\mu \cap \la} _\alpha$, and $\mr{e} _{ac} ^*$ intertwines $V_\la$ and $V_\mu | _{\mr{Int} _\mu (\la )}$:
$$ \mr{e} _{ac} ^* V_\la K_Z ^\la Z^*  =  V_\mu K^{\mu \cap \la} _Z Z^* \ = V_\mu \mr{e} _{ac} ^* K^\la _Z Z^*. $$ 
Since $\la$ is AC, the vector $I + N _\la \in F^2 _d (\la)$ is a WC vector and is in the range of a bounded intertwiner, $ Y : F^2 _d \rightarrow F^2 _d (\la )$, $ Y y = I + N_\la$ for some $y \in F^2 _d$. If the vector $y$ is not $L-$cyclic, then consider the $L-$invariant subspace
$$ F^2 _d y := \bigvee L^\alpha y. $$ (Here $\bigvee$ denotes closed linear span.)  Then, since $y$ is a cyclic vector for $L| _{F^2 _d y \otimes \C ^d }$, the NC Beurling Theorem \cite[Theorem 2.1]{DP-inv}, \cite[Theorem 2.3]{AriasPop} implies that 
$$ F^2 _d y = \ran{\Theta _y (R)}, $$ for some right-inner (isometric) $\Theta _y (R) \in R^\infty _d $.  Let $y' \in F^2 _d$ be such that $\Theta _y (R) y' = y$ and define 
$$ X:= \mc{C} _\mu ^* \mr{e} _{ac} ^* \mc{C} _\la Y \Theta _y (R), \quad \quad \dom{X} := \C \{ L_1 , ... , L_d \} y' \subseteq F^2 _d. $$ This operator is well-defined since:
\ba X L^\alpha y' & = &  \mc{C} _\mu ^* \mr{e} _{ac} ^* \mc{C} _\la  Y \Theta _y (R) L^\alpha y' \nn \\
& = & \mc{C} _\mu ^* \mr{e} _{ac} ^* \mc{C} _\la  Y L^\alpha y \nn \\
& = & \mc{C} _\mu ^* \mr{e} _{ac} ^* \mc{C} _\la (L^\alpha + N_\la) \nn \\
& = & \mc{C} _\mu ^* \mr{e} _{ac} ^* K_\alpha ^\la \nn \\
& = & \mc{C} _\mu ^* K_\alpha ^{\mu \cap \la}. \nn \ea
The operator $X$ is densely-defined since $y' \in F^2 _d$ must be $L-$cyclic: If $x \in F^2 _d$ is orthogonal to $\bigvee L^\alpha y'$ then $\Theta _y (R) x \perp \ran{\Theta _y (R)}$ so that $x \equiv 0$ since $\Theta _y (R)$ is an isometry. Finally, $X$ is also closable. This is a consequence of a general fact: if $T$ is a densely-defined closed operator, $C$ is a bounded operator, and $TC$ is densely-defined, then it is necessarily closed on 
$$\{ x \in \dom{C} | \ Cx \in \dom{T} \}. $$ Indeed, if $p_n (L) y' \in \dom{X}$, $p_n \in \fp$ is such that $p_n (L) y ' \rightarrow 0$ and $X p_n y \rightarrow g$, then since $Y' := \mc{C} _\la Y \Theta _y (R) $ is bounded, $Y ' p_n y '  \rightarrow 0 $. Since $\mr{e} _{ac} ^*$ is the adjoint of the closed operator $\mr{e} _{ac}$, it is closed, and since $y_n := Y' p_n y' \in \dom{e_{ac} ^*}$ obeys $y_n \rightarrow 0$, and $\mr{e} _{ac}^* y_n \rightarrow \mc{C} _\mu g$, it must be that $g =0$. This proves that $X$ is closable, and that $\fp y '$ is a core for its closure, $\ov{X}$, which is densely defined in $F^2 _d$.

For simplicity of notation, write $X$ in place of its closure, $\ov{X}$. One can check (using that $\mr{Int} _\mu (\la )$ is reducing for $V_\mu$) that $X$ intertwines $L$ and $\Pi _\mu$.  By Lemma \ref{ubtwine}, if $X \neq 0$, then $\ov{\ran{X}} \subseteq F^2 _d (\mu )$ is a non-empty $\Pi _\mu -$reducing subspace of weak$-*$ continuous vectors. Since $\mu$ is weak$-*$ singular, this is not possible and we conclude that $ \scr{H} ^+ (H _\mu ) \bigcap \scr{H} ^+ (H _\la ) = \{ 0 \}$.
\end{proof}

\begin{cor} \label{singacsing2}
Let $\mu \in \mr{Sing}(\A ^\dag ) _+$ be singular. If $\ga \in (\A ^\dag ) _+$ is such that $\ga \geq \mu$ has Lebesgue Decomposition $\ga = \ga _{ac} + \ga_s$ then  $\mu \leq \ga _s$.
\end{cor}

If $\mu , \la \in (\A ^\dag ) _+$ are NC measures so that $\la $ dominates $\mu$, $\mu \leq t^2 \la$ for some $t>0$, then the bounded operator $D_\mu := E_\mu ^* E_\mu$ is $\la-$Toeplitz (and has norm at most $t^2$), \emph{i.e.},
$$ \pi _\la (L_k) ^* D_\mu \pi _\la (L_j ) = \delta _{k,j} D_\mu, $$ and we have that
$$ \mu (a) = \ip{I + N_\mu }{D_\mu \pi _\la (a) (I + N_\mu ) } _{\la }; \quad \quad a \in \A. $$ The positive semi-definite operator $D_\mu$ will be called the Arveson-Radon-Nikodym derivative of $\mu$ with respect to $\la$. There is a special case where our Arveson-Radon-Nikodym derivative belongs to the commutant of the GNS representation $\pi _\mu$, this happens when $\Pi _\mu$ is a Cuntz row isometry (\emph{i.e.} if $\mu$ is a column-extreme NC measure, see \cite{JM-freeCE}):

\begin{lemma} \label{interCuntz}
Let $\Pi, \sigma$ be row isometries on $\H , \J$, respectively, and suppose that $X : \H \rightarrow \J$ is a bounded $(\Pi, \sigma )-$intertwiner, $X \Pi ^\alpha = \sigma ^\alpha X$.
If $\Pi$ is a Cuntz unitary then also
$$ X^* \sigma ^\alpha = \Pi ^\alpha X^*, $$ so that $D := X^* X$ belongs to the commutant of the von Neumann algebra generated by $\Pi$, $\mr{vN} (\Pi )$, and $D' = X X ^*$ belongs to the commutant of $\mr{vN} (\sigma )$.
\end{lemma}
\begin{proof}
Using that $\Pi $ is Cuntz, 
$$ X  =  X \Pi  \Pi  ^* = \sigma X \otimes I_d \Pi  ^*. $$ Hence,
$$ \sigma ^* X  =  \bpm \sigma _1 ^*  X \\ \vdots \\ \sigma _d  ^* X \epm 
= \sigma ^* \sigma (X \otimes I_d ) \Pi ^* 
=  \bpm X \Pi _1  ^* \\ \vdots \\ X \Pi _d ^* \epm. $$
This proves that $X \Pi _k  ^* = \sigma _k ^* X$, and taking adjoints yields the first claim: 
$$ \Pi ^\alpha X^* = X^* \sigma ^\alpha. $$ The commutation formulas are then easily verified:
$$ D \Pi ^\alpha  =  X^* X \Pi ^\alpha  =  X^* \sigma ^\alpha X  = \Pi ^\alpha D. $$ Since $D = X^* X \geq 0$, it follows that $D$ also commutes with $(\Pi ^\alpha ) ^*$. Similarly, 
$$ D ' \sigma ^\alpha  =  XX ^* \sigma ^\alpha = X \Pi ^\alpha X^* = \sigma ^\alpha D'. $$
\end{proof}

\begin{remark}
There is a theory of absolute continuity, Radon-Nikodym derivatives and Lebesgue Decomposition for completely positive operator-valued maps on a $C^*-$algebra initiated by Arveson \cite{ArvI,GK-ncld}. In this theory, if $\mu , \la$ are positive linear functionals on a $C^*-$algebra $\mc{E}$ and $\mu \leq \la$, then the Arveson-Radon-Nikodym derivative $D_\mu$, defined as above, always belongs to the commutant of the left regular GNS representation $\pi _\la$. 

In our theory, since $\A$ is not a $C^*-$algebra, this fails to be true in general.  If $\la$ is such that $\Pi _\la$ is not a Cuntz row isometry, and $\la \geq \mu$, the Arveson-Radon-Nikodym derivative $D_\mu$ is a positive semi-definite $\la-$Toeplitz contraction, but it is generally not in the commutant of $\Pi _\la$. For example, if $\la =m$, is NC Lebesgue measure and $\mu = m_x$ where $x = x(R) 1$ and $x(R) \in R^\infty _d$ is bounded, then $\mu$ is dominated by $m$ and the Arveson-Radon-Nikodym derivative $D_\mu = x(R) ^* x(R)$ is not in the commutant of $\mc{E} _d = C^* (I , L)$ where here we are identifying $\Pi _m \simeq L$. Indeed, the commutant of $C^* (I, L)$ is trivial. Nevertheless, we expect the Lebesgue decomposition theory for completely positive maps to play a role in the development of the Lebesgue decomposition of a positive NC measure $\mu$ with respect to an arbitrary positive NC measure $\la$.  
\end{remark}

\begin{proof}{ (of Corollary \ref{singacsing2})}
Since absolutely continuous and weak$-*$ continuous are the same, we have that $P_{ac} = P_L + P _{C-L} = P _{wc}$, where $P_L, P_{C-L}$ are the $\Pi _\ga -$reducing projections onto the type$-L$ and Cuntz type$-L$ subspaces of $F^2 _d (\ga )$. We first prove that $D_\mu P _{C-L} =0$. Define $E := E_\mu P_{C-L}$, and $D := E E^*$, a positive semi-definite contraction on $F^2 _d (\mu)$. Observe that $E : F^2 _d (\ga ) \hookrightarrow F^2 _d (\mu )$
intertwines the Cuntz unitary $\Pi _{C-L} := \Pi _\ga P _{C-L}$ and the row isometry $\Pi _\mu$. By Lemma \ref{interCuntz}, $D = E E^* = E_\mu P_{C-L} E_\mu ^*$ is in the commutant of the von Neumann algebra generated by $\Pi _\mu$, 
$$ D \Pi _\mu ^\alpha = \Pi _\mu ^\alpha D . $$ 
It follows that if we define
$$\varphi (L^\alpha ) := \ip{I + N _\mu }{D \Pi _\mu ^\alpha (I + N_\mu ) } _\mu, $$ then $\varphi \in (\A ^\dag )_+$ and $\varphi \leq \mu$. Indeed if $p (L) ^* p (L) = u (L)  ^* + u (L )$ for free polynomials $p, u$, then if we extend $\varphi$ to $\A ^*$ in the canonical way by 
$$ \varphi ( a ^* ) := \varphi (a) ^*, $$ then 
\ba \varphi (p ^* p) & = & \ov{\ip{I +N_\mu }{D u (\Pi _\mu ) (I + N_\mu )} _\mu} + \ip{I + N _\mu }{D u(\Pi _\mu ) (I + N_\mu )} _\mu \nn \\
& = & \ip{I + N_\mu}{ \left( u (\Pi _\mu ) ^* D + D u (\Pi _\mu ) \right) (I + N_\mu) } _\mu \nn \\
& = & \ip{I + N_\mu}{D \left( u (\Pi _\mu ) ^* + u (\Pi _\mu ) \right) (I + N _\mu )} _\mu \quad \quad \mbox{(Since $\Pi _\mu ^*$ commutes with $D$.)} \nn \\
& = & \ip{I+N_\mu }{D p (\Pi _\mu ) ^* p (\Pi _\mu ) (I + N _\mu )} _\mu \quad \quad \mbox{(Since $\Pi _\mu$ is a row isometry.)} \nn \\
& = & \ip{I + N _\mu }{p (\Pi _\mu ) ^* D p (\Pi _\mu ) (I+ N_\mu )}_\mu \geq 0. \nn \ea This proves that $\varphi$ is positive so that $\varphi \in (\A ^\dag ) _+$ is an NC measure. Also since $D$ is a positive contraction, it is clear that $\varphi \leq \mu$. However, by construction,
$$ \varphi (L^\alpha ) = \ip{P_{C-L} E_\mu ^* (I + N _\ga ) }{ \Pi _\ga ^\alpha P _{C-L} E_\mu ^* (I + N_\ga ) } _\ga, $$ is an absolutely continuous NC measure since any vector in the range of $P_{C-L}$ is an AC vector. Since $\varphi $ is also dominated by the singular NC measure $\mu$, Lemma \ref{singhered} implies that $\varphi \equiv 0$, and $P_{C-L} E_\mu ^* = 0$ so that $P_{C-L} D_\mu = P_{C-L} E_\mu ^* E_\mu =0$. \\

Now consider $\la := \ga _L$, the type$-L$ part of $\ga$. By Proposition \ref{WCisS}, we have that $\scr{H} ^+ (H _\la ) \bigcap \scr{H} ^+ (H_\mu ) = \{ 0 \}$. Define,
$$ \varphi (L ^\alpha ) := (\ga _L + \mu ) ( L ^\alpha ) = \ip{I + N _\ga }{(D_\mu + P_L ) \Pi _\ga ^\alpha (I + N _\ga )}. $$ 
It follows that $D_\varphi = D_\mu + P_L$ and $D_\varphi = E_\varphi ^* E_\varphi$ where $
E_\varphi = \mr{C} _\varphi ^* \mr{e} _{\varphi } ^* \mc{C} _\ga $ is a bounded embedding of norm at most $\sqrt{2}$, and $\mr{e} _{\varphi} : \scr{H} _+ (H _\varphi ) \hookrightarrow \scr{H} ^+ (H_\ga )$ is the bounded embedding of norm at most $\sqrt{2}$. However by Proposition \ref{WCisS},
$$ \scr{H} ^+ (H _\varphi ) \simeq \scr{H} ^+ (H _\mu ) \oplus \scr{H} ^+ (H_\la ), $$ so that 
$\mr{e} _\varphi \simeq \mr{e} _\mu \oplus \mr{e} _\la $ must be a contraction since both $\mr{e} _\mu , \mr{e} _\la$ are contractive embeddings. (Here, recall that we defined $\la := \ga _L$.) This proves that $D = D_\mu + P_L$ is a contraction. In particular, for any $x \in F^2 _d (\ga _L )  = F^2 _d (\la )$, 
$$ 0 \leq \ip{x}{Dx} _\ga = \ip{x}{D _\mu x } _\ga + \ip{x}{x} _\ga \leq \ip{x}{x} _\ga, $$ and this proves that $D _\mu P_L =0$. 

In conclusion, $D_\mu P_{ac} =0$ so that $D_\mu = D_\mu P_{ac} + D_\mu P_s = D_\mu P_s$, and $\mu \leq \ga _s$.
\end{proof}

\begin{cor} \label{singcone}
The sets $AC ( \A ^\dag ) _+$ and $\mr{Sing} (\A ^\dag ) _+$ of absoutely continuous and singular positive NC measures on the Free Disk System are positive hereditary cones.
\end{cor}
Recall that hereditary means that if $\la , \mu \in (\A ^\dag ) _+$, $\la \geq \mu$, and $\la$ is AC or singular, then $\mu$ is also AC or singular, respectively.
\begin{proof}
The set of positive absolutely continuous NC measures, $AC (\A ^\dag ) _+$, is a hereditary positive cone by Lemma \ref{domAC}. Lemma \ref{singhered} also proved that $\mr{Sing} (\A ^\dag ) _+$ is hereditary in $(\A ^\dag ) _+$, and it remains to show the set of singular NC measures is a positive cone. 

Suppose that $\mu _1 , \mu _2 \in \mr{Sing} (\A ^\dag ) _+$ are singular and let $\ga = \mu _1 + \mu _2$. Then by Corollary \ref{singacsing2}, if $P_{ac}, P_s$ denote the Lebesgue Decomposition reducing projections for $\Pi _\ga$,
$$ P_{ac} = (D_{\mu _1} + D_{\mu _2} ) P_{ac} = 0, $$ so that $I_{\ga} = P_s$ and $\ga = \mu _1 + \mu _2$ is singular.
\end{proof}
\begin{cor} \label{acsingchar}
Let $\mu \in (\A ^\dag ) _+$ be a positive NC measure. The following are equivalent:
\bn
    \item $\mu \in AC (\A ^\dag ) _+$ is absolutely continuous.
    \item $\scr{H} ^+ (H _\mu ) \bigcap H^2 (\B ^d _\N )$ is dense in $\scr{H} ^+ (H _\mu )$.
    \item The GNS row isometry $\Pi _\mu$ is weak$-*$ continuous, \emph{i.e.} the direct sum of type$-L$ and Cuntz type$-L$ row isometries. 
    \item Every vector $x \in F^2 _d (\mu )$ is a weak$-*$ continuous vector for $\mu$. 
    \item The quadratic form $q_\mu$ with form domain $\A 1 \subset F^2 _d$ is closable. 
\en
\end{cor}

\begin{cor}
Given an NC measure $\mu \in (\A ^\dag ) _+$, the following are equivalent:
\bn
    \item $\mu \in \mr{Sing} (\A ^\dag ) _+$ is singular.
    \item $\scr{H} ^+ ( H _\mu ) \bigcap H^2 (\B ^d _\N ) = \{ 0 \}.$
    \item $F^2 _d (\mu + m ) = F^2 _d (\mu ) \oplus F^2 _d$.
    \item $\Pi _\mu$ is the direct sum of dilation type and von Neumann type row isometries.
    \item $q_\mu$ with dense form domain $\A 1 \subseteq F^2 _d$ is a singular form.
\en
\end{cor}

\begin{cor} \label{ncldadd}
If $\mu, \la \in (\A ^\dag ) _+$ with (unique) Lebesgue decompositions $\mu = \mu _{ac} + \mu _s$, $\la = \la _{ac} + \la _s$ then 
$$ (\mu +\la) _{ac} = \mu _{ac} +\la _{ac}, \quad \mbox{and} \quad (\mu +\la )_s = \mu _s +\la _s. $$ 
\end{cor}
\begin{proof}
Set $\ga := \mu + \la = (\mu _{ac} + \la _{ac} ) + (\mu _s + \la _s ) =  \ga _{ac} + \ga _{s}$. Then by maximality $\mu _{ac} + \la _{ac} \leq \ga _{ac}$, and also by Corollary \ref{singcone} and Corollary \ref{singacsing2}, since $\mu _s + \la _s$ is singular, $\mu_{s} + \la _s \leq \ga _s$, and it follows that equality must hold in both cases.
\end{proof}

\section{Example: An NC measure of dilation type}

Recall that there is a natural bijection between (positive) NC measures and (right) NC Herglotz functions, $\mu \leftrightarrow H_\mu$. The transpose map $\dag$ also defines a natural involution which takes the right NC Herglotz class onto the left NC Herglotz class of all locally bounded NC functions in $\B ^d _\N$ with non-negative real part, see \cite[Section 3.9]{JM-NCFatou}. The Cayley Transform then implements a bijection between the left NC Schur class of contractive NC functions in $\B ^d _\N$ and and the left NC Herglotz class. If $\mu \in (\A ^\dag ) _+$ is the (essentially) unique NC measure corresponding to the contractive NC function $B \in [ H^\infty (\B ^d _\N ) ]_1$, we write $\mu = \mu _B$, and $\mu _B$ is called the NC Clark measure of $B$, see \cite[Section 3]{JM-NCFatou} for details. 

By \cite[Corollary 7.25]{JM-NCFatou}, if $B \in \scr{L} _d$ is inner, then its NC Clark measure is singular, so that its GNS representation $\Pi _B := \Pi _{\mu _B}$ is a Cuntz row isometry which can be decomposed as the direct sum of a dilation-type row isometry and a von Neumann type row isometry. 

Classically, any sum of Dirac point masses is singular with respect to Lebesgue measure on the circle. Motivated by this,  consider the positive linear functional $\mu \in ( \mc{A} _2 ^\dag ) _+$ defined by 
$$ \mu (L ^\alpha ) = \left\{ \begin{array}{cc} 0 & 2 \in \alpha \\ 1 & 2 \notin \alpha \end{array} \right. ; \quad \quad \alpha \in \F ^2,  $$ and $\mu (I) = 1$. This is a `Dirac point mass' at the point $(1,0) \in \B ^2 _1 $ on the boundary, $\partial \B^2 _\N$ of the NC unit ball. Results of Popescu imply that since $Z := (1, 0 )$ is a row contraction, that the map $\mu$, sending 
$$ L^\alpha \mapsto Z ^\alpha, $$ extends to a positive linear functional on $\mc{A} _2$ \cite[Theorem 2.1]{Pop-NCdisk}. 

\begin{claim}
$L_2 +N_\mu$ is a wandering vector for $\Pi _\mu$ and $\Pi _\mu$ has vanishing von Neumann part.
\end{claim}
\begin{proof}
Note that any wandering vector for $\Pi _\mu$ is always a weak$-*$ continuous vector. Indeed, if $w$ is wandering for $\Pi _\mu$, then 
\ba \mu _w (L^\alpha ) & = & \ip{w}{\Pi_\mu ^\alpha w}_{F^2 _d (\mu)}  \nn \\
& = & \| w\| ^2 _{F^2 _d (\mu)} \delta _{\alpha, \emptyset} = \| w \| ^2 m (L^\alpha), \nn \ea is a constant multiple of NC Lebesgue measure, and hence belongs to $WC(\A ^\dag ) _+$. 

To see that $L_2 + N_\mu$ is wandering for $\Pi _\mu$, calculate
$$ \ip{L_2 +N_\mu }{\Pi _\mu ^\alpha (L_2 + N_\mu )} _\mu = \mu ( L_2 ^* L^\alpha L_2 ) = \delta _{\alpha, \emptyset}. $$ However $L_2 + N_\mu$ is also cyclic for $\Pi _\mu$ since $\pi _\mu (L_2 ) ^* (L_2 +N_\mu) =I + N_\mu$ which is cyclic for $\Pi _\mu$. This means that the smallest reducing subspace which contains the AC vector $L_2 + N_\mu$ is all of $F^2 _d (\mu)$, so that $F^2 _d (\mu _{vN} ) = \{ 0 \}$. 
\end{proof}

\begin{claim} The NC measure $\mu$ is the NC Clark measure of $B_\mu (Z) = Z_1$, a left-inner NC function.  Hence $\Pi _\mu$ is purely of dilation-type.
\end{claim}
\begin{proof}
    The (left) Herglotz function, $H_\mu$ of $\mu$ is:
\ba H_\mu (Z) & = & ( \mr{id} _n \otimes \mu ) \left( ( I_n \otimes I _{F^2} +ZL^* ) ( I_n \otimes I_{F^2}  - ZL ^* ) ^{-1} \right) \nn \\
& = & 2 \sum _{\alpha} Z^\alpha \mu \left( L ^{\alpha ^\dag}  \right) ^* -I_n \nn \\
& = & 2 \sum _{k=0} ^\infty Z_1 ^k  - I_n = 2 (I - Z_1 ) ^{-1} - I_n \nn \\
& = & (I + Z_1) (I -Z_1 ) ^{-1}. \nn \ea 
It follows that the Cayley Transform, $B_\mu $, of $H_\mu$ is $B_\mu (Z) = Z_1$, which is inner.
By \cite[Corollary 7.25]{JM-NCFatou}, $\Pi _\mu$ is the direct sum of a dilation-type and a von Neumann-type row isometry, and the previous claim shows that the von Neumann part vanishes.
\end{proof}

%\bibliography{Bibs/ncld}

\end{document}